\documentclass[a4paper,11pt]{amsart}
\usepackage[left=2.7cm,right=2.7cm,top=3.5cm,bottom=3cm]{geometry}

\usepackage{amsthm,amssymb,amsmath,amsfonts,mathrsfs,amscd,dsfont}
\usepackage[T1]{fontenc}
\usepackage{inputenc}
\usepackage[all,cmtip]{xy}
\usepackage{latexsym}
\usepackage{longtable}
\usepackage{hyperref}

\numberwithin{equation}{section}

\newfont{\cyr}{wncyr10 scaled 1100}
\newfont{\cyrr}{wncyr9 scaled 1000}

\theoremstyle{plain}
\newtheorem{theorem}{Theorem}[section]
\newtheorem{proposition}[theorem]{Proposition}
\newtheorem{lemma}[theorem]{Lemma}

\theoremstyle{definition}
\newtheorem{definition}[theorem]{Definition}
\newtheorem{assumption}[theorem]{Assumption}

\theoremstyle{remark}
\newtheorem{remark}[theorem]{Remark}

\newcommand{\Q}{\mathbb Q}
\newcommand{\N}{\mathds N}
\newcommand{\Z}{\mathbb Z}
\newcommand{\R}{\mathbb R}
\newcommand{\C}{\mathbb C}

\DeclareMathOperator{\End}{End}
\DeclareMathOperator{\Aut}{Aut}

\DeclareMathOperator{\Gal}{Gal}
\DeclareMathOperator{\GL}{GL}

\DeclareMathOperator{\Ind}{Ind}

\newcommand{\ord}{\mathrm{ord}}
\newcommand{\new}{\mathrm{new}}

\newcommand{\rec}{\mathrm{rec}}
\newcommand{\frob}{\mathrm{frob}}

\newcommand{\pwseries}[1]{[\![#1]\!]}

\newcommand{\Sha}{\mbox{\cyr{X}}}

\usepackage[usenames]{color}
\definecolor{Indigo}{rgb}{0.2,0.1,0.7}
\definecolor{Violet}{rgb}{0.5,0.1,0.7}
\definecolor{White}{rgb}{1,1,1}
\definecolor{Green}{rgb}{0.1,0.9,0.2}

\newcommand{\longmono}{\mbox{\;$\lhook\joinrel\longrightarrow$\;}}

\newcommand{\longepi}{\mbox{\;$\relbar\joinrel\twoheadrightarrow$\;}}

\newcommand{\smallmat}[4]{\bigl(\begin{smallmatrix}#1&#2\\#3&#4\end{smallmatrix}\bigr)}

\newfont{\gotip}{eufb10 at 12pt}

\newcommand{\p}{\mathfrak{p}}

\include{thebibliography}

\begin{document}

\title[Anticyclotomic $p$-adic $L$-functions for Hida families]{An explicit comparison of anticyclotomic $p$-adic $L$-functions for Hida families}
\today
\date{}
\author{Chan-Ho Kim}
\address{(Chan-Ho Kim) Center for Mathematical Challenges, Korea Institute for Advanced Study, 85 Hoegiro, Dongdaemun-gu, Seoul 02455, Republic of Korea}
\email{chanho.math@gmail.com}
\author{Matteo Longo}
\address{(Matteo Longo) Dipartimento di Matematica, Universit\`a di Padova, Via Trieste 63, 35121 Padova, Italy}
\email{mlongo@math.unipd.it}

\thanks{}

\begin{abstract}
The aim of this note is to compare several anticyclotomic $p$-adic $L$-functions for modular forms and $p$-adic families of ordinary modular forms, which have been defined and studied from different perspectives by Skinner--Urban, Hida, Perin-Riou, Bertolini--Darmon, Vatsal, Chida--Hsieh, Longo--Vigni, Castella--Longo and Castella--Kim--Longo. The main result of this paper is a comparison between the central critical twist  of the two-variable anticyclotomic $p$-adic $L$-function obtained as specialisation of the three-variable $p$-adic $L$-function of Skinner--Urban and the two-variable $p$-adic $L$-function introduced by one of the authors in collaboration with Vigni by means of $p$-adic families of Gross points. 
\end{abstract}



\maketitle

\tableofcontents

\section{Introduction} 
\subsection{Overview}
Anticyclotomic $p$-adic $L$-functions attached to modular forms and $p$-adic families of modular forms have a long story and have been studied in many papers by several authors; among others, Hida \cite{Hida1}, Perrin-Riou \cite{PR}, Vatsal \cite{Vat1}, Bertolini--Darmon \cite{BD-Heegner, BDmumford-tate, bdIMC}, and, more recently, Skinner--Urban \cite{SU}, Lei--Loeffler--Zerbes \cite{LLZ}, Kings--Loeffler--Zerbes \cite{KLZ}, 
B\"{u}y\"{u}kboduk--Lei \cite{BL1}, \cite{BL2}, Longo--Vigni \cite{LV-MM}, \cite{LV-Pisa}, 
Castella--Longo \cite{cas-longo}, Castella--Kim--Longo \cite{CKL}, and, from a slightly different perspective, in Longo-Nicole \cite{LN}. 
However, it is rather unclear whether these constructions give us the \emph{entirely} same object since the choice of periods and the fudge factors in the interpolation formulas of these $p$-adic $L$-functions would depend on their own constructions.

To the best knowledge of the authors, the following two constructions have been  mainly used. 
\begin{itemize}
\item $p$-adic Rankin--Selberg convolution (\cite{Hida1, PR} for a single modular form, \cite{SU, Hida2} for families).
\item Gross points (\cite{BD-Heegner, BDmumford-tate, bdIMC, Vat1} for a single modular form, \cite{LV-MM, cas-longo, CKL} for families).
\end{itemize}

Our main goal is to clarify the relation between these different constructions of $p$-adic $L$-functions.
We start by recalling the three-variable $p$-adic $L$-function of Skinner--Urban \cite{SU}, and their two-variable and one-variable anticyclotomic specialisations; this $p$-adic $L$-function is constructed, as mentioned before, by means of $p$-adic Rankin--Selberg products for families of modular forms. We then recall the construction of the $p$-adic $L$-function via Gross points, following the description given by Chida--Hsieh \cite{ChHs1}, which generalises the constructions of Vatsal and Bertolini--Darmon.
In Theorem \ref{SU-CH}, we give a detailed proof of the well-known relation between the one-variable anticyclotomic specialisation of Skinner--Urban's $p$-adic $L$-function and Chida--Hsieh's $p$-adic $L$-function for the case of weight $2$ modular forms.
We then recall Hida--Perrin-Riou's one-variable anticyclotomic $p$-adic $L$-function, following the construction of B\"{u}y\"{u}kboduk--Lei \cite{BL1}, and compare it with (a twist of) the one-variable anticyclotomic Skinner--Urban $p$-adic $L$-function (Theorem \ref{HBL-SU}, which can also be obtained from results of \cite{BL1} via a more indirect method). 

The main result of this paper is to give a more detailed and complete proof of \cite[Theorem 5.3]{CKL} in which we compare the two-variable anticyclotomic specialisation of the central critical twist of the Skinner--Urban $p$-adic $L$-functions and the two-variable anticyclotomic $p$-adic $L$-function constructed by one of the authors of this paper in collaboration with Vigni in \cite{LV-MM} by means of compatible families of Gross points. 

\subsection{Statement of the main results} \label{subsec:main-result}

Our main result applies to $2$-variables $p$-adic $L$-functions associated with 
the primitive branch $\mathbb{I}$ of $p$-adic Hida families, of tame level $N$ and trivial character, whose residual representation $\bar\rho$ satisfies suitable arithmetic assumptions. We clarify the set of assumption that we need by specifying the properties satisfied of one of the arithmetic specialisations of $\mathbb{I}$.

Let $N \geq 1$ be an integer and $p\nmid 6N$ a prime number. Fix embeddings $\iota_p: \bar\Q\hookrightarrow\bar\Q_p$ and $\iota_\infty: \bar\Q\hookrightarrow \C$.
Let 
\[g=\sum_{n\geq 1}a_nq^n \in S_k(\Gamma_0(Np))\]
be a weight $k$ modular form of level $\Gamma_0(Np)$. 
Let $L / \Q_p$ be a finite field extension which contains $\iota_p(a_n)$ for all Fourier coefficients $a_n$ of $g$, and let $\rho:G_\Q\rightarrow\GL_2(L)$ 
be the $p$-adic representation attached to $g$, where $G_\Q=\Gal(\bar{\Q}/\Q)$. Choose a Galois-stable $\mathcal{O}_L$-lattice $T$, where $\mathcal{O}_L$ is the valuation ring of $L$; then 
$\rho$ is isomorphic to a representation (denoted with the same symbol)   
$\rho:G_\Q\rightarrow\GL_2(\mathcal{O}_L)$ with values in $\GL_2(\mathcal{O}_L)$. 
Denote by \[\bar\rho:G_\Q\longrightarrow\GL_2(k_L)\] the associated residual representation, where  
$k_L$ is the residue field of $\mathcal{O}_L$. The representation $\bar\rho$ \emph{depends} on the choice of the lattice $T$, but its semisimplification does not, so if $\bar\rho$ is irreducible, then it its isomorphism class does not depend on the choice of $T$. 

Let $K/\Q$ be a quadratic imaginary field of discriminant $D$ such that $(Np,D)=1$. Factor $N=N^+N^-$ where a prime number $\ell$ divides $N^+$ (respectively, $N^-$) if and only if $\ell$ is split 
(respectively, inert) in $K$. We say that $N=N^+N^-$ is the \emph{factorisation of $N$ associated with $K$}. 

\begin{assumption}\label{ass} The modular form $g=\sum_{n\geq 1}a_nq^n\in S_k(\Gamma_0(Np))$ (normalised with $a_1=1$), and the quadratic imaginary field $K$ satisfy the following conditions: 
\begin{enumerate}
\item $k\geq 2$ is an even integer.
\item $N\geq 1$ and $p\nmid 6N$. 
\item $g$ has trivial character. 
\item $g$ is a \emph{$p$-stabilized newform}: either $g$ is an ordinary newform of level $\Gamma_0(Np)$ or $g$ is the ordinary $p$-stabilisation of an ordinary newform $g_0$ of level $\Gamma_0(N)$.
\item $g$ is \emph{$p$-ordinary}: $\iota_p(a_p)$ is a $p$-adic unit. 
\item The discriminant $D$ of $K$ is prime to $Np$; denote by $N=N^+N^-$ the factorisation of $N$ associated with $K$. 
\item $N^-$ is a square-free product of an odd number of primes. 
\item $\bar\rho$ is irreducible. 
\item $\bar\rho$ is ramified at all primes dividing $N^+$ which are congruent to $1$ modulo $p$. 
\item $\bar\rho$ is ramified at all primes dividing $N^-$.
\item $p$ splits in $K$. 
\end{enumerate}
\end{assumption}

\begin{remark} 
Let $\bar\rho \vert_{G_{\Q_p}}$ be the restriction of $\bar\rho$ to $G_{\Q_p} = \Gal(\bar{\Q}_p/\Q_p)$, and denote by
$(\bar\rho \vert_{G_{\Q_p}} )^{ss}$ the semisimplification of $\bar\rho \vert_{G_{\Q_p}}$.
Recall that $\bar\rho$ is said to be 
\emph{ordinary} if $(\bar\rho \vert_{G_{\Q_p}} )^{ss} \simeq \varepsilon_1  \oplus \varepsilon_2$ for characters $\varepsilon_1$, $\varepsilon_2$.
Also, an ordinary representation $\bar\rho$ is said to be \emph{distinguished} if  $\varepsilon_1\not=\varepsilon_2$ in the above decomposition (see \cite[\S3.3.5]{SU}).
The conditions in Assumption \ref{ass} imply that $\bar\rho$ is both ordinary and distinguished.\end{remark}

Let
\[\mathbf{f}=\sum_{n\geq 1}\mathbf{a}(n)q^n \in\mathbb I\pwseries{q}\] be the primitive branch of the $p$-adic family of ordinary modular forms of tame level $N$ passing through $g$; so 
$\mathbb{I}$ is a local reduced finite integral extension of the Iwasawa algebra $\mathcal{O}_L\pwseries{W}$ (where $W$ is a formal variable) and there exists an arithmetic morphism $\phi:\mathbb{I}\rightarrow \mathcal{O}_L$ such that 
\[\phi(\mathbf{f})=\sum_{n\geq 1}\phi(\mathbf{a}_n)q^n=g.\]

Let $\Gamma_K^-\simeq\Z_p$ be the Galois group of the anticylotomic $\Z_p$-extension of $K$. 
Denote by $S$ the set of rational primes dividing $NDp$. We consider the following $p$-adic $L$-functions: \begin{itemize}
\item $L^\mathrm{SU}_{S,\mathbb{I}}\in \mathbb{I}\pwseries{\Gamma_K^-}$: this is the central critical twist  of the two-variable anticyclotomic specialisation of the three-variable $p$-adic $L$-function of Skinner--Urban (see Definition \ref{def su} for details on the construction of this $p$-adic $L$-function, and especially on the role of the critical twist). 
\item $L^\mathrm{LV}_{S,\mathbb{I}}\in  \mathbb{I}\pwseries{\Gamma_K^-}$: this is the two-variable $p$-adic $L$-function constructed by means of Gross points in \cite{LV-MM} and studied in \cite{cas-longo} and \cite{CKL}, with Euler factors at all primes dividing $S$ removed (see \S \ref{quaternionic section} for more details).  
\end{itemize}

\begin{theorem} \label{mainthm}
Under Assumption $\ref{ass}$, we have 
$(L^\mathrm{SU}_{S,\mathbb{I}})=(L^\mathrm{LV}_{S,\mathbb{I}})$ as ideals in 
$\mathbb{I}\pwseries{\Gamma_K^-}$.  
\end{theorem}

Applications of this result to some conjectures stated in \cite{LV-MM} on the relation between $ L^\mathrm{LV}_\mathbb{I}$ and the characteristic ideal of the Selmer group of Hida's big ordinary Galois representation attached to $\mathbf{f}$ is given in \cite{KL}. 
The proof of Theorem \ref{mainthm} exploits the interpolation formulas for special values of complex $L$-functions enjoyed by these $p$-adic $L$-functions. However, a direct comparison between the interpolation properties of $L^\mathrm{SU}_{S,\mathbb{I}}$ and $L^\mathrm{LV}_{S,\mathbb{I}}$ is not available, because these two functions interpolate 
different special values. Our strategy is to first compare these functions with other anticyclotomic $p$-adic $L$-functions introduced by Chida-Hsieh and B\"{u}y\"{u}kboduk--Lei, and then use the resulting relations to prove Theorem \ref{mainthm}.

As a general notation, we denote $f=\phi(\mathbf{f})$ the specialisations of $\mathbf{f}$ at an arithmetic morphism $\phi:\mathbb{I}\rightarrow\bar\Q_p$, leaving the symbol $g$ for our fixed modular form as before, which will be one of these specialisations; when $f=\phi(\mathbf{f})\in S_k(\Gamma_0(Np))$, we will usually denote $f_0$ the modular form of level $\Gamma_0(N)$ whose ordinary $p$-stabilisation is $f$ when $f$ is old at $p$, and let $f_0=f$ when $f$ is new at $p$; observe that for $k\geq 4$, any ordinary $p$-stabilised newform $f$ of weight $k$, level $\Gamma_0(Np)$ and trivial character is the ordinary $p$-stabilisation of a form $f_0\in S_k(\Gamma_0(N))$ of weight $k$, level $\Gamma_0(N)$ and trivial character.  
 
\subsection*{Acknowledgments} The authors would like to thank the referee for his/her careful reading of the paper, which led to a correction of a mistake in the previous version of the paper and to a significative improvement of the exposition.
The authors would like to thank F. Castella, M.-L. Hsieh, K. B\"{u}y\"{u}kboduk and A. Lei for useful discussions.  
C.-H. K. was partially supported 
by a KIAS Individual Grant (SP054103) via the Center for Mathematical Challenges at Korea Institute for Advanced Study and
by the National Research Foundation of Korea (NRF) grant funded by the Korea government(MSIT) (No. 2018R1C1B6007009). M. L. was partially supported by  PRIN 2017 \emph{Geometric, algebraic and analytic methods in arithmetic} and INDAM GNSAGA. 

\section{Hecke characters}

\subsection{Algebraic Hecke characters} 

Let $K$ be an imaginary quadratic field of discriminant $D$.
Denote by $\mathfrak{d}$ the different ideal of $K$, by $\mathcal{O}_K$ the ring of its algebraic integers. For any place $v$ of $K$, let $K_v$ be the completion of $K$ at $v$ and denote $\mathcal{O}_{K_v}$ the valuation ring of $K_v$. For any finite place $v$ of $K$, denote $\mathfrak{p}_v$ the associated prime ideal and let $\varpi_v$ be a uniformiser of $\mathcal{O}_{K_v}$.   
Let $\mathbb{A}_K^\times$ be the idele group of $K$ and write an element $x$ of $\mathbb{A}_K^\times$ as $x=(x_v)_v$ where $v$ ranges over all valuations of $K$ and $x_v\in K_v$. Write $\mathbb{A}_{K,f}^\times$ for the finite adeles of $\mathbb{A}_K^\times$, and write an element $x=(x_v)_v$ as $x=(x_f,x_\infty)$ with $x_f\in\mathbb{A}_{K,f}^\times$ and $x_\infty\in K_\infty^\times=(K\otimes_\Q \R)^\times$. 

A \emph{Hecke character} \[\chi:\mathbb{A}_K^\times/K^\times\longrightarrow\C^\times\] of $K$ is a continuous group homomorphism 
$\chi:\mathbb{A}_K^\times\rightarrow\C^\times$ (denoted with the same symbol $\chi$) that is trivial on the image of $K^\times$ into $\mathbb{A}_K^\times$ via the diagonal embedding $x\mapsto (x_v)_v$ with $x_v=x$ for all $v$. For any place $v$ of $K$
write $\chi_v:K_v^\times\rightarrow\C^\times$ for the restriction of $\chi$ to the image of $K_v$ into $\mathbb{A}_K^\times$ via the map $x\mapsto (x_w)_w$ with $x_w=1$ if $w\neq v$ and $x_v=x$. We often write $\chi=\bigotimes_v\chi_v$.

Let $\chi:\mathbb{A}_K^\times/K^\times\rightarrow\C^\times$ be a Hecke character. We say that $\chi$
is \emph{algebraic} if the restriction $\chi_\infty:K_\infty^\times\rightarrow\C^\times$ of $\chi$ to the infinity component $K_\infty^\times$ of $\mathbb{A}_K^\times$ has the form $\chi_\infty(z)=z^{t_1}\bar{z}^{t_2}$ for a pair of integers $(t_1,t_2)\in\Z\times\Z$; in this case, 
we say that the algebraic Hecke character $\chi$ has \emph{infinity type} $(t_1,t_2)\in\Z\times\Z$.
We say that an algebraic Hecke character $\chi$ is \emph{unitary} if $\chi:\mathbb{A}_K^\times\rightarrow \mathbb{S}^1$ where $\mathbb{S}^1$ is the unit circle, 
and that $\chi$ is  \emph{anticyclotomic} if $\chi(x)=1$ for all $x\in\mathbb{A}_\Q^\times$.
Note that unitary algebraic Hecke characters have infinity type $(m,-m)$ for some integer $m$. Let $w_\chi=t_1+t_2$ be the \emph{weight} of $\chi$. 
The \emph{conductor} of an algebraic Hecke character $\chi$ is the $\mathcal{O}_K$-ideal $\mathfrak{f}_\chi=\prod_v\p_v^{e_v}$ (the product is over all finite places of $K$) where $e_v\in\Z$  
is the smallest non negative integer such that $\chi_v$ is trivial on $1+\p_v^{e_v}$ (note that $e_v=0$ for all $v$ except possibly a finite number, and therefore the definition of $\mathfrak{f}_\chi$ makes sense).

For each prime number $\ell$, fix an embedding 
$\iota_\ell: \bar\Q\rightarrow\bar\Q_\ell$. 
Let $\chi:\mathbb{A}_K^\times/K^\times\rightarrow\C^\times$ be an algebraic Hecke character of infinity type $(t_1,t_2)$. 
Denote by $\alpha_\ell^{(t_1,t_2)}:K_\ell^\times\rightarrow\bar\Q_\ell^\times$ the continuous character uniquely determined by the condition 
$\alpha_\ell^{(t_1,t_2)}(x\otimes 1)=\iota_\ell^{t_1}(x)\iota_\ell^{t_2}(\bar{x})$, 
where $K_\ell=K\otimes_\Q\Q_\ell$ and $x\mapsto \bar{x}$ denotes the 
action of the non-trivial automorphism of $\Gal(K/\Q)$. 
Define the \emph{$\ell$-adic avatar} 
\[\hat\chi_\ell:
\mathbb{A}_{K,f}^\times/K^\times\longrightarrow\bar\Q_\ell^\times\] of $\chi$ to be the continuous character $\hat{\chi}_\ell=\chi\cdot\alpha_\ell^{(t_1,t_2)}$. Since 
$\Q_\ell^\times$ is a totally disconnected topological space, $\ker(\hat\chi_\ell)$ contains the connected component of $1$ in $\mathbb{A}_{K,f}^\times$, and therefore gives rise to a continuous character 
\[\sigma_{\chi,\ell}:\Gal(K^\mathrm{ab}/K)\longrightarrow \bar\Q_\ell^\times\]
such that $\sigma_{\chi,\ell}\circ\rec_K=\hat\chi_\ell$, where $K^\mathrm{ab}$ is the maximal abelian extension of $K$, and  
\[\mathrm{rec}_ K :\mathbb A_ K ^\times/ K ^\times\longepi
\Gal( K ^\mathrm{ab}/ K )\] is the geometrically normalized 
reciprocity map of class field theory (so $\mathrm{rec}_ K(\varpi_v)=\frob_v$ where $\frob_v$ is a geometric Frobenius element at $v$). We may also view 
$\sigma_{\chi,\ell} :G_K \rightarrow\bar\Q^\times_\ell$ as a character of 
the absolute Galois group $G_K=\Gal(\bar{K}/K)$ of $K$ via the canonical projection $G_ K \twoheadrightarrow G_ K ^\mathrm{ab}$, 
where $G_K^\mathrm{ab}=\Gal(K^\mathrm{ab}/K)$; also, 
$\sigma_{\chi,\ell}$ takes values in $E_{\chi,\ell}^\times$, where $E_{\chi,\ell}$ is a finite field extension of $\Q_\ell$. 
See \cite[Ch. I, \S2.3]{serre-abelian}, or \cite[Lecture 3, \S 3]{rohrlich-root-numbers}.

\begin{remark}
When we consider the interpolation formula for anticyclotomic $p$-adic $L$-functions of modular forms, the $p$-adic characters are assumed to be locally algebraic as reviewed in $\S$\ref{CH}. Note that the complex avatars of $p$-adic locally algebraic anticyclotomic characters are algebraic Hecke characters.
\end{remark}

\subsection{Gauss sums} Let $\chi=\bigotimes\chi_v:\mathbb{A}_K  ^\times\rightarrow\C^\times $ be an algebraic Hecke character. 
For each prime $\ell$ of $\Q$, we write $ \chi_\ell$ for the character 
of $K _\ell^\times$  given by the product of 
the characters $ \chi_v$ for $v\mid \ell$, 
where $K _\ell=K \otimes_\Q\Q_\ell=\prod_{v\mid\ell}K _v$.
Denote by $c_{ \chi_\ell}$ the conductor ideal of $ \chi_\ell$, where, following \cite[\S8.1.1]{SU}, this is an ideal in $\mathcal{O}_\ell=\mathcal{O}_ K \otimes_\Z \mathds{Z}_\ell$; so when $\ell$ splits in $K$,  $c_{\chi_\ell}$ is identified with a pair of residue classes of integers in $\mathcal{O}_\ell\simeq\Z_\ell\times\Z_\ell$. Let $d_\ell$ be a generator of the ideal $\mathfrak{d}\mathcal{O}_\ell$. Let 
$e_\infty(x)=e(x)=e^{2\pi i x}$ and $e_\ell(1/\ell)=e(-1/\ell)$ for rational primes $\ell$ be the standard additive characters (\cite[\S8.1.2]{SU}). 
The \emph{local Gauss sum} of $ \chi_\ell$ is given by 
\[\mathfrak{g}( \chi_\ell,c_{ \chi_\ell} d_\ell)=\sum_{a\in(\mathcal{O}_\ell/c_{ \chi_\ell})^\times} \chi_\ell(a)\cdot e_\ell\left(\mathrm{Tr}_{K_\ell/\Q_\ell}\left(\frac{a}{c_{ \chi_\ell}d_\ell}\right)\right),\]
and the \emph{global Gauss sum} of $\chi$ is defined by 
\[\mathfrak{g}( \chi)=\prod_\ell \chi_\ell^{-1}(c_{ \chi_\ell}d_\ell)\cdot \mathfrak{g}( \chi_\ell,c_{ \chi_\ell}d_\ell).\]

\begin{lemma}\label{gauss} 
Let $\chi$ be an anticyclotomic unitary Hecke character of conductor $p^n\mathcal{O}_K$.
Assume   $p=\p\bar\p$ is split in $ K  $.Then $\mathfrak{g}( \chi)=\pm p^{n}$. \end{lemma}

\begin{proof}
We first claim that $ \chi _\p= \chi _{\bar{\p}}^{-1}$.
Let $\mathbf{x} = (1, \cdots, 1, x, 1, \cdots )\in\mathbb{A}_{\Q}^\times$ where $x\in\Q_p^\times$ and all other components of $\mathbf{x}$ are equal to $1$.
Denote by $\mathbf{x}_K$ the image of $\mathbf{x}$ in $\mathbb{A}_K^\times$.
Then both components of $\mathbf{x}_K=(x_v)$ at $\p$ and ${\bar{\p}}$ are $x$ and all other components of $\mathbf{x}_K$ equal to $1$.
Since $\chi$ is anticyclotomic, $\chi$ is trivial on $\mathbb{A}_\Q^\times$.
Thus, we have 
\[ \chi _p(\mathbf{x}_K)= \chi _{\p}(x)\cdot \chi _{\bar{\p}}(x)=\prod_{w\nmid p} \chi _w^{-1}(1)=1 . \]
 The claim now follows from the isomorphism 
$\mathcal{O}_p/p^n\mathcal{O}_p\simeq \Z/p^n\Z\oplus \Z/p^n\Z$ 
since $p$ splits in $K$.
In particular, since $\chi$ is unitary,
we have  
 \begin{equation}\label{gauss-1}\mathfrak{g}(\bar\chi_\p)=\mathfrak{g}( \chi_{\bar\p})
\end{equation}
where we write $\bar\chi_\p(x)=\overline{\chi_\p(x)}$, 
$\mathfrak{g}( \chi_\p)=\mathfrak{g}( \chi_\p,p^n)$ and 
$\mathfrak{g}( \chi_{\bar\p})=\mathfrak{g}( \chi_{\bar\p},p^n)$. 

We now compute the Gauss sum. Since the conductor of $ \chi$ is $p^n$, all terms $\mathfrak{g}( \chi_\ell,c_{ \chi_\ell}d_\ell)$ with $\ell\neq p$ are equal to $1$. 
Also, $ \chi_\ell(p^nd_\ell)=1$ for all $\ell\neq Dp$. If $v\mid \ell$ is a prime ideal of $ K  $ dividing $\mathfrak{d}$, where $\ell$ is a rational prime, then $(d_\ell^2)=(\ell)$. Since $ \chi$ is anticyclotomic, this forces $ \chi_v(d_\ell)=\pm1$. 
Finally, we study the local Gauss sum at $p$. Since $p\nmid D$, we have $(d_p)=(1)$, and, using that $p$ splits in $K$, one easily shows that 
\begin{equation}\label{gauss0}
\mathfrak{g}( \chi)=\mathfrak{g}( \chi_\p)\cdot\mathfrak{g}( \chi_{\bar\p}).\end{equation}
By \cite[Proposition 2.2(ii)]{Martinet}, using that $\mathfrak{f}_\chi=p^n$, we have 
\begin{equation}\label{martinet}
\mathfrak{g}( \chi_\p)\cdot\mathfrak{g}(\bar \chi_\p)= \chi_\p(-1)\cdot p^n,\end{equation} and combining \eqref{gauss-1}, \eqref{gauss0} and \eqref{martinet} completes the proof because $ \chi_\p^2(-1)= \chi_\p(1)=1$.   
\end{proof}

\section{Complex $L$-functions} 

\subsection{$L$-functions of modular forms}
Let 
$f=\sum_{n\geq 1}a_nq^n \in S_k(\Gamma_0(N))$ be a normalized eigenform where $q=e^{2\pi i z}$.
The complex $L$-function of $f$
$$L(f,s)=\sum_{n=1}^\infty\dfrac{a_n}{n^s}$$
converges absolutely for $s \in\C$ with $\Re(s)>k/2+1$, satisfies a functional equation, and extends to an entire function.
{In each domain $\Re(s)\geq k/2+1+\delta$ for $\delta>0$}, 
$L(f,s)$ admits the Euler product
\[L(f,s)=\prod_\ell  L_\ell (f,s).\]
where 
$L_\ell (f,s)=(1-a_\ell \ell^{-s}+  \ell^{k-1-2s})^{-1}$
and $\ell$ runs over all primes.
We refer to \cite[$\S$5.9 and 5.10]{DS}) for details.

\subsection{$L$-functions of Hecke characters} 
Let $K/\Q$ be a quadratic imaginary field of discriminant $D$, and let 
$\chi:\mathbb A_ K^\times/K^\times\rightarrow \mathbb \C^\times$ be an algebraic Hecke character of infinity type $(m,n)$. 
Denote by $\mathfrak f_\chi$ conductor of $\chi$.  
Write $\N=\N_{K/\Q}$ for the norm operator on ideals of $K$; we shall adopt the same symbol for the corresponding map on ideles, and for the norm map in Galois theory, but the context will clarify the meaning.  
Recall the compatible system of Galois representations 
$\{\sigma_{\chi,\ell}\}_\ell$ associated with $\chi$ and 
define the $L$-function of $\chi$ for $s\in\C$ with $\Re(s)>1+w_\chi/2$ to be 
\[L(\chi,s)=\prod_{v\nmid \mathfrak{f}_\chi}L_v(\chi,s)\]
where  
$L_v(\chi,s)=(1-\sigma_{\chi,\ell}(\frob_v)\N(\p_v)^{-s})^{-1}$ and $v$ runs over all finite places of $K$ which do not divide the conductor $\mathfrak{f}_\chi$ of $\chi$; here note that, since $\{\sigma_{\chi,\ell}\}$ is a strictly compatible system of Galois representations, and $\sigma_{\chi,\ell}$ is unramified at all $v\nmid \mathfrak{f}_\chi$, we have 
$\sigma_{\chi,\ell}(\frob_v)\in\bar\Q$ for all such $v$, and we view $\sigma_{\chi,\ell}(\frob_v)\in\C$ via the fixed embedding $\iota_\infty:\bar\Q\hookrightarrow\C$. The complex function $L(\chi,s)$ of $s$ can be extended to a meromorphic  function on the whole $\C$, satisfying a functional equation which relates $L(\chi,s)$ and $L(\bar{\chi},w_\chi+1-s)$, where $\bar{\chi}(x)=\overline{\chi(x)}$, and here $x\mapsto \bar{x}$ denote the complex conjugation; see \cite[Chapter 0, \S6]{Shappacher} or \cite[Lecture 3, \S3]{rohrlich-root-numbers} for details. 

\subsection{Theta series}\label{sectheta}
As in the previous \S, let $K/\Q$ be a quadratic imaginary field of discriminant $D$. Let 
$\chi:\mathbb A_ K^\times/K^\times\rightarrow \mathbb \C^\times$ be an algebraic Hecke character of conductor $\mathfrak{f}_\chi$ and infinity type $(m,0)$. 
We put 
\[g_\chi(z):=\sum_{\mathfrak a}\chi(\mathfrak a)e^{2\pi i \N(\mathfrak a)z}\] 
where the sum runs over all integral ideals $\mathfrak{a}$ of $K$.
The formal series above can also be written 
as 
$g_\chi(q)= \sum_{n\geq 0}b_nq^n$  
with 
$b_n=\sum_{\N(\mathfrak a)=n}\chi(\mathfrak a)$. 
Then $g_\chi$ is a modular form of weight $m+1$, level $M=|D|\cdot\N(\mathfrak{f}_\chi)$ and character $\psi_{g_\chi}$ defined for $x\in \Z$ by 
\[\psi_{g_\chi}(x)=\chi_{D}(x)\chi((x))\mathrm{sgn}(x)^M\]
where $\chi_D$ is the quadratic character associated to $K$.
 Also, $g_\chi$ is a cusp form unless $m=0$ and $\chi$ is the composition of $\N$ with a Dirichlet character. See \cite[Theorem 4.8.2]{Miy}.

\begin{remark}\label{remark-infinity-types}
The formulation in \cite[Theorem 4.8.2]{Miy} and other references, \emph{e.g.} \cite[\S12.3]{Iwaniec}, is slightly different. 
One fixes a Hecke character $\xi$ such that $\xi((a))=(a/|a|)^u$ for $a\equiv 1\pmod{\mathfrak{f}}$ and considers the modular form 
$\sum_{\mathfrak a}\xi(\mathfrak a)\N(\mathfrak a)^{u/2}e^{2\pi i \N(\mathfrak a)z}$.
Since such a $\xi$ has infinity type $(u/2,-u/2)$, we see that 
$\chi=\xi\cdot\N^{u/2}$ has infinity type $(u,0)$, and when $g_\chi$ is a cusp form, we have \[L(g_\chi,s)=L(\xi,s+u/2).\] \end{remark}

By \cite[p. 214]{Iwaniec} and taking into account the previous Remark \ref{remark-infinity-types}, $g_\chi$ satisfies the following transformation formula: 
\[(g_{\chi})_{|m+1}W_{|D|\cdot\mathrm{N}(\mathfrak{f}_\xi)}= \frac{\mathfrak{g}(\xi)}{i^{2m-1}\N(\mathfrak{f}_\xi)^{1/2}}\cdot g_{\bar\chi},\]
where $\mathfrak{f}_\xi$ is the conductor of $\xi$ and $W_{|D|\cdot\mathrm{N}(\mathfrak{f}_\xi)}=\smallmat 0{-1}{|D|\cdot\mathrm{N}(\mathfrak{f}_\xi)}{0}$ is the Atkin--Lehner involution; note that in \emph{loc. cit.} the action of a matrix $\gamma=\smallmat abcd\in\GL_2(\R)$ with positive determinant on a modular form $f$ is via the formula \[f_{|\ell}\gamma(z)=\det(\gamma)^{\ell/2}(cz+d)^{-\ell}f(\gamma(z)).\]
In particular, if  
$\xi$ is an anticyclotomic unitary Hecke character of conductor $p^n\mathcal{O}_K$, and $p=\p\bar\p$ is split in $ K  $, then
\[\frac{\mathfrak{g}(\xi)}{i^{2m-1}\N(\mathfrak{f})^{1/2}}=\pm i^{2m-1}.\] 

\subsection{Twisted $L$-functions} 
Let $\chi:\mathbb A^\times_K/K^\times\rightarrow \C^\times$ be an algebraic Hecke character. Let $V_{\chi,\ell}$ denote the one-dimensional $E_{\chi,\ell}$-vector space affording the 
character $\sigma_{\chi,\ell}$; recall that $E_{\chi,\ell}$ is a finite extension of $\Q_\ell$ (\cite[Ch. 0, \S5]{Shappacher}).  
Consider the induced representation 
\[\Ind_{\chi,\ell}=\Ind_{G_K}^{G_\Q}(\sigma_{\chi,\ell})
:=V_{\chi,\ell}\otimes_{\Z[G_K]}\Z[G_\Q].\]
Since $[K:\Q]=2$, this is a 2-dimensional $E_{\chi,\ell}$-vector space 
which can be explicitly described as a direct sum $M_1\oplus M_2$ with 
$M_1$ and $M_2$ both isomorphic to $V_{\chi,\ell}$ as $G_K$-modules, 
and such that the non-trivial element $\tau$ of $\Gal(K/\Q)$ permutes the two components 
$M_1$ and $M_2$ in the sense that $\tau(x,0)=(0,x)$ for all $x\in M_1$.
See \cite[Ch. III, \S 5]{Br}.

Let $F=\Q(a_n:n\geq 1)$ be the Hecke field of $f$, and $\mathcal{O}_F$ its ring of integers.
For each prime ideal $\lambda$ of $\mathcal{O}_F$ lying above a rational prime $\ell$, 
consider the $\lambda$-adic modular Galois representation 
\[\rho_{f,\lambda}:G_\Q \longrightarrow\GL_2(F_{\lambda})\]
attached to $f$
 where $F_\lambda$ is the completion of $F$ at $\lambda$. 
Then $\lbrace \rho_{f,\lambda} \rbrace_{\lambda}$ forms a compatible system of Galois representations.

Let $L=L(\rho_{f,\lambda},\sigma_{\chi,\ell})$ be a finite extension of $\Q_\ell$ containing $F_\lambda$ and $E_{\chi,\ell}$. By extension of scalars, we can view  
$\rho_{f,\lambda}$ (respectively, $\Ind_{\chi,\ell}$) as a representation of $G_{\Q}$ on the two-dimensional $L$-vector spaces $W_f=V_{f,\lambda}\otimes_{F_\lambda}L$, where $V_{f,\lambda}$ is the representation space of $\rho_{f,\lambda}$  
(respectively, the two-dimensional $L$-vector space $W_\chi=V_{\Ind_{\chi,\ell}}\otimes_{E_{\chi,\ell}}L$, where $V_{\Ind_{\chi,\ell}}$ is the representation space of $\Ind_{\chi,\ell}$). 
We define  the tensor product representation $\rho_{f,\lambda}\otimes \Ind_{\chi,\ell}$ of $G_ \Q$ as follows: put 
$W_{f\otimes\chi}=W_f\otimes_L W_\chi$, which is a 4-dimensional $L$-vector space.
Define the representation
\[\rho_{f,\lambda}\otimes\Ind_{\chi,\ell}  : G_ \Q  \to \Aut_{L}(W_{f\otimes\chi})\] 
by
\[(\rho_{f,\lambda}\otimes\Ind_{\chi,\ell})(g):=\rho_{f,\lambda}(g)\otimes\Ind_{\chi,\ell}(g)\] for any $g\in G_ \Q$.

Define the $L$-function of $f\otimes\chi$, 
following \emph{e.g} 
\cite[Lecture 3, \S3]{rohrlich-root-numbers}, so follows: Let $W_{f\otimes\chi}$ be as above, for fixed $\lambda\mid \ell$, and 
for each prime number $q\neq\ell$, one may define the local Euler factors 
\[P_q(X)=\det\left(1-X \cdot \mathrm{frob}_q|(W_{f\otimes\chi})^{I_q} \right)\] 
where $I_q$ is the inertia subgroup at $q$ and $\frob_q$ 
is the geometric Frobenius at $q$.
Then $P_q(X)$ is independent of the choice of $\lambda\mid \ell$, and (by varying $\ell$) we may 
define 
\[L_q(f, \chi,s)=P_q(q^{-s})^{-1}\] 
for each prime $q$, and set 
\[L(f, \chi, s)=\prod_{q}L_q(f,\chi,s)\]
where the product is over all rational primes $q$.

\subsection{Explicit Euler factors} \label{Euler factors}
We derive explicit Euler factors in the case when $\chi$ has infinity type $(m,0)$ and corresponds to a theta series $g_\chi$. 
In terms of Euler factors of $L(f,s)$ and $L(g_\chi,s)$, 
one may express the above $L$-series as follows. First, write the Euler factors 
of $f$ and $g_\chi$ as 
\[1-a_q(f) q^{-s}+ q^{k-1-2s}=(1-\alpha_q(f) q^{-s})(1-\beta_q(f) q ^{-s})\] 
and
\[1-a_q(g_\chi) q^{-s}+\psi_{g_\chi}(q)q^{m-2s}=(1-\alpha_q(g_\chi) q ^{-s})(1-\beta_q(g_\chi) q ^{-s})\]
where recall that the weight of $g_\chi$ is $m+1$. The $L$-functions are 
\[L(f,s)=\prod_{q}(1-\alpha_q(f) q ^{-s})^{-1}(1-\beta_q(f) q ^{-s})^{-1}\]
and
\[L(g_\chi,s)=\prod_{q}(1-\alpha_q(g_\chi)  q ^{-s})^{-1}(1-\beta_q(g_\chi) q ^{-s})^{-1}.\] 

The complex function
\[D(f,g_\chi,s):=\sum_{n\geq 1}\frac{a_n(f) \cdot a_n(g_\chi)}{n^s}\] 
converges absolutely for $\Re(s)$ sufficiently large, and extends to an entire function. 
If we put 
\begin{align*}
D_q (f,g_\chi,s) = \ & (1-\alpha_q(f)\beta_q(f)\alpha_q(g_\chi)\beta_q(g_\chi) q^{-2s}) \\
& \cdot (1-\alpha_q(f) \alpha_q(g_\chi) q ^{-s})^{-1} \cdot (1-\beta_q(f) \alpha_q(g_\chi) q ^{-s})^{-1} \\
& \cdot (1-\alpha_q(f) \beta_q(g_\chi) q ^{-s})^{-1} \cdot (1- \beta_q(f) \beta_q(g_\chi) q ^{-s})^{-1} ,
\end{align*}
then $D(f,g_\chi,s)$ admits the Euler product  
\[D(f,g_\chi,s)=\prod_{q }D_q (f,g_\chi,s).\]
See \cite[Lemma 15.9.4]{Jacquet} or \cite[\S 3]{Shi-comm} and \cite[\S 3]{Shi}. 
We also observe that
\begin{align*}
1-\alpha_q(f)\beta_q(f)\alpha_q(g_\chi)\beta_q(g_\chi) q^{-2s} & = 1-\psi_{g_\chi}(q)q^{k-1+(m+1)-1-2s} \\
& =L_q(\psi_{g_\chi},2s+2-(k+m+1))^{-1}
\end{align*}
where the weight of $f$ is $k$ and the weight of $g_\chi$ is $m+1$. 
By \cite[Sec. 15]{Jacquet}, we have  
\begin{align*}
L_q(\rho_{f,\lambda}\otimes\sigma_{\chi,\ell},s) = \ & (1-\alpha_q(f) \alpha_q(g_\chi) q ^{-s})^{-1} \cdot (1-\beta_q(f) \alpha_q(g_\chi) q ^{-s})^{-1} \\
& \cdot (1-\alpha_q(f) \beta_q(g_\chi)  q ^{-s})^{-1} \cdot (1-\beta_q(f) \beta_q(g_\chi) q ^{-s})^{-1}.
\end{align*}
Therefore, we have 
\[D_q(f,g_\chi,s)=L_q(\psi_{g_\chi},2s+2-(k+m+1)) \cdot L_p(f, \chi,s),\]
hence 
\[D(f,g_\chi,s)=L(\psi_{g_\chi},2s+2-(k+m+1))\cdot L(f, \chi,s).\]

\begin{remark}We have 
\[L_q (f, \chi , s)=\prod_{v\mid q }L_v(f, \chi,s) \] 
with 
\[L_v(f, \chi,s):=(1-\chi(v) \cdot \alpha_{\N(v)}(f)  \cdot \N(v)^{-s})^{-1}\cdot (1-\chi(v)  \cdot \beta_{\N(v)}(f)  \cdot \N(v)^{-s})^{-1}
\]
where $v$ is a finite place of $K$ dividing $q$, $\N(v)=\N(\p_v)$, 
$\chi(v) = \hat\chi_\ell(\varpi_v)$ and $\alpha_{p^j}(f):=( \alpha_p(f) )^j$, $\beta_{p^j}(f):=(\beta_p(f))^j$.
\end{remark}

\subsection{Partial $L$-functions}
If $S$ is a finite number of primes, one may define 
\[L^S(f, \chi,s):=\prod_{q\not\in S} L_q(f, \chi,s).\]
Obviously, one has (in the domain of absolute convergence) 
\[L(f, \chi,s)=L^S(f, \chi, s)\cdot\left(\prod_{q\in S}L_q(f, \chi,s) \right).
\] Since each of the three factors can be extended to an holomorphic function, the principle of 
identity of analytic functions tells us that the above equality also holds on all the domain $\C$ 
of convergence. We will often write 
\[L_K^S(f,\chi,s)=L^S(f,\chi,s)\]
to emphasise on the LHS the role of $K$, which is implicit (incorporated in $\chi$) in the notation on the RHS.

\section{Skinner--Urban $p$-adic $L$-functions} \label{sec:skinner-urban}

\subsection{Hida families} \label{subsec:skinner-urban-hida-families}
Let 
$\mathbf{f}=\sum_{n\geq 1}\mathbf{a}(n)q^n \in\mathbb I\pwseries{q}$ 
be the primitive branch of the Hida family passing through the modular form $g$ fixed as in \S\ref{subsec:main-result} and satisfying Assumption \ref{ass}. 
In particular, recall that the character of $\mathbf{f}$ is trivial and $\mathbb{I}$ is a local reduced finite integral extension of $\Lambda_W=\Z_p\pwseries{W}$, where $W$ is an indeterminate, \emph{cf.} \cite[\S3.3.9]{SU} (note that in the notation of \emph{loc. cit.} we take $\chi$ to be the trivial character).    

We fix the notation and terminology from \cite{SU}. 
Let $\Gamma_ K \simeq\Z_p^2$ be the maximal $\Z_p$-extension of $ K $, and denote by $\Gamma_ K ^+\simeq\Z_p$ the cyclotomic $\Z_p$-extension of $ K $ and $\Gamma_ K ^-\simeq\Z_p$ the anticyclotomic $\Z_p$-extension of $ K $. Define the Iwasawa algebras $\Lambda^\pm_ K =\Z_p[\![\Gamma_ K ^\pm]\!]$ and 
$\Lambda_ K =\Z_p[\![\Gamma_ K ]\!]$. Fix an isomorphism $\Z_p[\![\Gamma_ K ^\pm]\!]\simeq\Z_p[\![W]\!]$ sending a topological generator $\gamma_\pm$ of $\Gamma_ K ^\pm$ to $W+1$.

An \emph{arithmetic character} 
$\phi:\mathbb{I}\rightarrow\bar\Q_p$ is a continuous $\Z_p$-algebra 
map whose restriction to 
$\Lambda_W$ satisfies $\phi(1+W)=\zeta(1+p)^{k_\phi-2}$, where $\zeta$ is a primitive $p^{t_\phi-1}$-root of unity, for some integer $t_\phi\geq 1$, and $k_\phi$ is an integer; 
we call $t_\phi$ the \emph{level} of $\phi$ and $k_\phi$ the \emph{weight} of $\phi$ (see \cite[\S3.3.8]{SU}). Any arithmetic $\phi:\mathbb{I}\rightarrow\bar\Q_p$ corresponds to an eigenform
\begin{equation}\label{f_phi}\mathbf{f}_\phi\in S_{k_\phi}(\Gamma_0(N)\cap\Gamma_1(p^{t_\phi}),\chi_\phi)\end{equation}
of weight $k_\phi$ and character $\chi_\phi$ defined as follows. Write
$\chi_\zeta:\Z_p[\![W]\!]\rightarrow \bar\Q_p$ for the arithmetic charachter defined by $1+W\mapsto\zeta$.
Then 
$\chi_\phi=\chi_{\zeta^{-1}}$, so that $\chi_\phi(\gamma_+)=\zeta^{-1}$. 
We finally view $\phi_\phi$ as a Dirichlet character of $(\Z/p^{t_\phi}\Z)^\times$ as follows. Fix an isomorphism $\Gamma_ K ^+\simeq 1+p\Z_p$. 
Identify $(\Z/p^{t_\phi}\Z)^\times$ with $(\Z/p\Z)^\times\times\left((1+p\Z_p)/(1+p\Z_p)^{t_\phi-1}\right)$ and let $\chi_\phi:(\Z/p^{t_\phi}\Z)^\times\rightarrow\bar\Q_p^\times$ 
to be trivial on $(\Z/p\Z)^\times$ and $\chi_\phi$ on $(1+p\Z_p)/(1+p\Z_p)^{t_\phi-1}$. 

Let $\mathbb{I}_K=\mathbb{I}[\![{\Gamma_K}]\!]$ and $\mathbb{I}_ K^\pm=\mathbb{I}[\![\Gamma_ K^\pm]\!]$.
 An \emph{arithmetic character} of $A$ for $A=\mathbb{I}_K$ or $A=\mathbb{I}_K^\pm$ is a continuous $\bar\Q_p$-valued $\Z_p$-algebra map whose restriction to $\mathbb{I}$ is an arithmetic character and such that 
$\phi(\gamma_+)=\zeta_+(1+p)^{k_{\phi_{|\mathbb{I}}}-2}$ and $\phi(\gamma_-)=\zeta_-$,  
where $\zeta_+$ and $\zeta_-$ are $p$-power roots of unity.
 
Let
$\boldsymbol{\varepsilon}_ K $ be the  
canonical character defined by composition 
\[\boldsymbol{\varepsilon}_ K :\Gal(\bar{ K }/ K )\longrightarrow \Gamma_ K \longmono \Lambda_ K ^\times\]
of the canonical projection and the inclusion of group-like elements;
then $\boldsymbol{\varepsilon}_ K$ factors through $\Gal(K^\mathrm{ab}/K)$, where $ K ^\mathrm{ab}$ is the maximal abelian extension of $ K $.  Define similarly the character 
\[\boldsymbol{\varepsilon}_ K ^+:\Gal( K ^\mathrm{ab}/ K )\longepi \Gamma_ K ^+\longmono(\Lambda_ K ^+)^\times.\] 
Let
\[\mathrm{rec}_ K :\mathbb A_ K ^\times/ K ^\times\longrightarrow 
\Gal( K ^\mathrm{ab}/ K )\] be the geometrically normalized 
reciprocity map of class field theory (so $\mathrm{rec}_ K (\ell)= \mathrm{frob}_\ell$, where
$\mathrm{frob}_\ell$ is the geometric Frobenius element). 
Define the characters
\[\mathit{\Psi}_ K  :\mathbb A_ K ^\times/ K ^\times
\overset{\mathrm{rec}_ K }\longrightarrow \Gal( K ^\mathrm{ab}/ K )
\overset{\boldsymbol{\varepsilon}_ K }
\longrightarrow \Lambda_ K ^\times,\] 
\[\mathit{\Psi}_ K ^+ :\mathbb A_ K ^\times/ K ^\times
\overset{\mathrm{rec}_ K }\longrightarrow \Gal( K ^\mathrm{ab}/ K )\overset{\boldsymbol{\varepsilon}_ K^+}
\longrightarrow (\Lambda_ K ^+)^\times.\] 
Let 
$\xi_\phi=\phi\circ(\Psi_ K /\Psi_K^+)$ be the composition (\cite[top of page 47]{SU}): 
\[\xi_\phi:\mathbb A_ K ^\times/ K ^\times
\overset{\mathrm{rec}_ K }\longrightarrow \Gal( K ^\mathrm{ab}/ K )
\overset{\boldsymbol{\varepsilon}_ K /\boldsymbol{\varepsilon}_ K ^+}
\longrightarrow \Lambda_ K ^\times\overset\phi\longrightarrow \bar\Q_p^\times.\] Note that $\xi_\phi$ is a finite order idele class character. 
Finally, define $\theta_\phi=\chi_\phi^{-1}\cdot\xi_\phi$; we denote $\mathfrak{f}_{\theta_\phi}$ the conductor of $\theta_\phi$. Note that $\theta_\phi$ is a finite order idele class character.

\subsection{Period integrals, canonical periods, and $L$-values of modular forms} \label{sec4.2}
We first review period integrals and the canonical periods following \cite[$\S$3.3.3]{SU}.
Recall that the Eichler--Shimura period map
$$\mathrm{Per} : S_k(N) \longrightarrow \mathrm{H}^1(\Gamma_1(N), \mathrm{Sym}^{k-2}(\mathbb{C}^2))$$
is defined by putting
$\mathrm{Per}(f(z))$ equal to the cohomology class represented by the cocycle 
$$\gamma \longmapsto \int^{\gamma(\tau)}_{\tau} f(z) (z^{k-1}, z^{k-2}, \cdots, 1) dz,$$
where the integration is over any path between $\tau$ and $\gamma(\tau)$.
It is well-known that the period map is Hecke-equivariant.

Let $f \in S_k(\Gamma_1(N))$ be an eigenform, and denote by $\mathbb{Q}(f)$ the field extension of $\mathbb{Q}$ by adjoining the Fourier coefficients of $f$.
Define
$\mathbb{Z}(f)_{(p)} :=\mathbb{Q}(f) \cap \iota^{-1}_p (\overline{\mathbb{Z}}_p)$, where $\overline{\Z}_p$ is the valuation ring of $\overline{\Q}_p$ and $\iota_p:\Q\hookrightarrow\bar{\Q}_p$ is the chosen embedding, 
and
$$M(f)_{(p)} := \mathrm{H}^1\left(\Gamma_1(N), \mathrm{Sym}^{k-2}( \mathbb{Z}(f)^2_{(p)} )\right)[\wp_f]$$
where $\wp_f$ is the height one prime ideal of the full Hecke algebra over $\mathbb{Z}_{(p)}$ faithfully acting on $S_k(N)$ corresponding to $f$.
We now assume that $f$ is a newform.
Then $M(f)_{(p)} $ is free of rank two over $\mathbb{Z}(f)_{(p)}$.
By using $\iota_{\infty} : \mathbb{Z}(f)_{(p)} \hookrightarrow \mathbb{C}$, we regard
$M(f)_{(p)}$ as a submodule of $\mathrm{H}^1(\Gamma_1(N), \mathrm{Sym}^{k-2}(\mathbb{C}^2))$
which spans the two-dimensional complex vector space 
$\mathrm{H}^1(\Gamma_1(N), \mathrm{Sym}^{k-2}(\mathbb{C}^2))[\wp_f]$.
Fix a $\mathbb{Z}(f)_{(p)} $-basis $(\gamma^+, \gamma^-)$ of $M(f)_{(p)}$ such that 
$\iota ( \gamma^{\pm} ) = \pm \gamma^{\pm}$ where $\iota$ is the involution associated with the conjugate action of $\left( \begin{smallmatrix}
1 & 0 \\ 0 & -1
\end{smallmatrix} \right)$ on the cohomology group 
$\mathrm{H}^1(\Gamma_1(N), \mathrm{Sym}^{k-2}(\mathbb{C}^2))$.
We define the canonical periods $\Omega^{\pm}_f \in \mathbb{C}^\times$ of $f$ by the equality
$$\mathrm{Per}(f(z)) = \Omega^+_f \cdot \gamma^+ + \Omega^-_f \gamma^-$$
and these periods are well-defined up to units in $\mathbb{Z}(f)_{(p)}$.
See Vatsal's paper 
\cite{Vatsal-period} for more details.

The connection between period integrals and $L$-values is as follows.
For an eigenform $f \in S_k(N)$, we have identity
$$2 \pi i \cdot \int^{0}_{i \infty}f(z)z^j dz = \dfrac{j!}{(-2\pi i)^j} \cdot L(f, j+1)$$
for $0 \leq j \leq k-2$.

\subsection{Three-variable $p$-adic $L$-functions} 
For the notational compatibility, 
denote by \[L_K^S(f,\xi,s)=L^S(f,\xi,s)\] as in \cite[\S 3.4.2]{SU} the $L$-function $L^S(f,\xi,s)$
where $\xi$ is a finite order Hecke character.

Under the condition that $p$ is split in $K$ and the 
assumptions $(\mathbf{irred})_\mathbf{f}$ 
and $(\mathbf{dist})_\mathbf{f}$ (\cite[\S3.3.10]{SU}), 
Skinner and Urban prove in  \cite[Theorem 12.6]{SU} the existence of 
$\mathcal L^S_{\mathbf{f}, K}\in 
\mathbb I_ K$, where $S$ is a finite set of primes containing all those primes dividing $DNp$, satisfying the following interpolation formula for the complex $L$-function $L_K^S(\mathbf{f}_\phi,\theta_\phi,s)$: for all 
arithmetic characters 
$\phi:\mathbb I_K\rightarrow \bar\Q_p$
we have
\begin{equation}\label{interpolation SU}
\phi(\mathcal L^S_{\mathbf{f},K}) =
\frac{u_{\mathbf{f}_\phi}\cdot ((k_\phi-2)!)^2\cdot\mathfrak{g}(\theta_\phi)\cdot\N(\mathfrak{f}_{\theta_\phi}\mathfrak{d})^{k_\phi-2}}
{a(p,\mathbf{f}_\phi)^{\ord_p\left(\N(\mathfrak{f}_{\theta_\phi})\right)}}
\cdot \frac{L^S_K(\mathbf{f}_\phi,\theta_\phi,k_\phi-1)}{(-2\pi i)^{2k_\phi-2}\cdot\Omega_{\mathbf{f}_\phi}^+\cdot\Omega_{\mathbf{f}_\phi}^-}\end{equation}
where
\begin{itemize}
\item  $u_{\mathbf{f}_\phi}$ is a unit which depends only on 
$\mathbf{f}_\phi$. 
\item $\Omega_{\mathbf{f}_\phi}^\pm$ are the canonical periods of $\mathbf{f}_\phi$, which are defined in \S\ref{sec4.2} up to $p$-adic units. 
\end{itemize}

\subsection{Critical twists} 
Howard in the papers \cite{howard-invmath} and \cite[$\S$2]{howard-annalen} introduces several characters $\Theta$, $\theta$ and 
$\Theta_\p$, $\theta_\p$, $\chi_\p$ for $\mathfrak p\subseteq\mathbb I$ for any 
arithmetic prime (i.e. $\mathfrak p=\ker(\phi_{|\mathbb I})$ for some arithmetic character $\phi:\mathbb{I}_K\rightarrow\bar\Q_p$). We recall the relevant definitions. 
Decompose the cyclotomic character
\[
\chi_{\rm cyc}:G_\Q \longrightarrow \Z_p^\times\simeq\boldsymbol{\mu}_{p-1}\times\Gamma
\]
 (thus, here we use the unadorned symbol $\Gamma$ for $1+p\Z_p$; also, 
in \cite{SU} the cyclotomic character is denoted $\epsilon$, while in 
\cite{howard-invmath} and \cite[$\S$2]{howard-annalen} it is denoted by $\chi_\mathrm{cyc}$) 
into its tame part
$\omega:G_\Q\rightarrow \boldsymbol{\mu}_{p-1}$
and its wild part $\chi_{\rm w}:G_\Q\rightarrow \Gamma$. The choice of
a square-root $\omega^{(k-2)/2}:G_\Q\rightarrow \boldsymbol{\mu}_{p-1}$
of $\omega^{k-2}$, which we fix from now on, determines the choice
of a \emph{critical character}
\[
\Theta :G_\Q\longrightarrow \Z_p\pwseries{\Gamma}^\times \simeq \Lambda_W^\times\subseteq \mathbb I^\times
\]
defined by 
$\Theta=\omega^{(k-2)/2}\cdot[\chi_{\rm w}^{1/2}],$
where $\chi_{\rm w}^{1/2}$ is the unique square-root of $\chi_{\rm w}$
with values in $\Gamma$
and $z\mapsto [z]$ is induced by the inclusion of group-like elements
$\Gamma\hookrightarrow\Z_p\pwseries{\Gamma}$ followed 
by the isomorphism $\Z\pwseries{\Gamma}\simeq \Lambda_{W}=\Z_p\pwseries{W}$
which takes $1+p$ to $1+W$. 
Let 
$
\theta:\Z_p^\times\rightarrow\mathbb I^\times
$
be such that $\Theta=\theta\circ\chi_{\rm cyc}$, and
for each arithmetic prime $\p\subseteq\mathbb I$, let
$
\theta_\p:\Z_p^\times\rightarrow \Q_p^\times
$
be the map induced by the composition of $\theta$ with $\phi_{|\mathbb I}$. 
In the same way, define 
$\Theta_\p:G_\Q\rightarrow \bar\Q_p^\times$
as the composition of $\Theta$ with  $\phi_{|\mathbb I}$. Finally, recall that for all arithmetic $\mathfrak p\subseteq \mathbb I$ of weight 2 we have 
$\chi_{0,\mathfrak p}=\theta_\mathfrak p^{-2}$ where $\chi_{0,\mathfrak p}$ is the restriction of $\chi$ 
to $\mathbb A_\mathbb Q^\times$. 

We introduce the character $\mathbf{\Theta}$
on the Galois group $\Gamma_K$ by the following diagram:  
\[
\xymatrix{
\mathrm{Gal}(K^{\mathrm{ab}}/K)  \ar@{^{(}->}[r] \ar@{->>}[d] & \mathrm{Gal}(K^{\mathrm{ab}}/\mathbb{Q}) \ar@{->>}[r] & \mathrm{Gal}(\mathbb{Q}^{\mathrm{ab}}/\mathbb{Q}) \ar[d]^-{\chi_{\rm cyc}} \ar[ld]_-{\Theta} \\ 
\Gamma_K \ar[r]^-{\mathbf{\Theta}} 
\ar@/_1pc/[rr]
& \mathbb{I}^\times & \mathbb{Z}^\times_p \ar[l]_-{\theta} & .
}
\]
Define the twist operator 
\[\mathrm{Tw}_{\mathbf{\Theta}^{-1}}:
\mathbb I\pwseries{\Gamma_K }\longrightarrow \mathbb I\pwseries{\Gamma_K }\] 
as the unique $\mathbb{I}$-algebra homomorphism characterised by sending $\gamma\in\Gamma_K$ to 
\[\mathrm{Tw}_{\mathbf{\Theta}^{-1}}(\gamma):= {\mathbf{\Theta}^{-1}}(\gamma)\cdot\gamma.\]

\begin{definition}\label{def su} Define the \emph{critical twist} of the Skinner--Urban three-variable $p$-adic $L$-function to be 
$\mathcal L^{S,\dagger}_{\mathbf{f},K }:=
\mathrm{Tw}_{\mathbf{\Theta}^{-1}}(\mathcal L^S_{\mathbf{f},K }).$
\end{definition}

\subsection{Anticyclotomic specialisations} 
The map $\Gamma_K\rightarrow \Gamma_K^-$ defined by $\gamma_+\mapsto1$ and $\gamma_-\mapsto\gamma_-$ induces a surjective map 
$\pi_\mathbb{I}^\mathrm{ac}:\mathbb{I}\pwseries{\Gamma_K}\twoheadrightarrow\mathbb{I}\pwseries{\Gamma_K^-}$.  

\begin{definition} The \emph{two-variable anticyclotomic specialisation} of the critical twist Skinner-Urban three-variable $p$-adic $L$-function is 
$L^\mathrm{SU}_{S,\mathbb{I}}=\pi_\mathbb{I}^\mathrm{ac}(\mathcal L^{S,\dagger}_{\mathbf{f},K })$. 
\end{definition}

Write $\mathbf{f}_{\phi}=f$ and consider the morphism
$\phi_{|\mathbb{I}}:\mathbb{I}\rightarrow\mathcal{O}$, where $\mathcal{O}$ is the valuation ring of a finite extension of $\Q_p$, containing the ring $\Z_p[a(n,f)]$.
Similarly as above, the map $\Gamma_K\rightarrow \Gamma_K^-$ defined by $\gamma_+\mapsto1$ and $\gamma_-\mapsto\gamma_-$ induces a surjective map
$\pi_\mathcal{O}^\mathrm{ac}:\mathcal{O}\pwseries{\Gamma_K}\twoheadrightarrow\mathcal{O}\pwseries{\Gamma_K^-}$.
 Define the morphism 
 $(\phi_{|\mathbb{I}}\otimes\mathrm{id}):\mathbb{I}\otimes_{\Z_p}\Z_p[\![\Gamma_K ]\!]\rightarrow \mathcal{O}[\![\Gamma_K ]\!]$ by $(\phi_{|\mathbb{I}}\otimes\mathrm{id})(x\otimes a)=\phi_{|\mathbb{I}}(x)\otimes a$; composing with the canonical isomorphism $\mathbb{I}[\![\Gamma_K ]\!]\simeq\mathbb{I}\otimes_{\Z_p}\Z_p[\![\Gamma_K ]\!]$ we thus obtain a morphism 
\[\pi_f^\mathrm{ac}:\mathbb{I}[\![\Gamma_K ]\!]\overset{\phi_{|\mathbb{I}}\otimes\mathrm{id}}\longrightarrow \mathcal{O}[\![\Gamma_K ]\!]
\overset{\pi_\mathcal{O}^\mathrm{ac}}\longrightarrow\mathcal{O}\pwseries{\Gamma_K^-}.\]

\begin{definition} The \emph{one-variable anticyclotomic specialisation} of the critical twist Skinner-Urban three-variable $p$-adic $L$-function is 
$L^\mathrm{SU}_{S,f}=\pi_f^\mathrm{ac}(\mathcal L^{S,\dagger}_{\mathbf{f},K })$. 
\end{definition}

\begin{remark} See also \cite[\S 3.4.6]{SU} for similar definitions; we also note that 
$\phi_{|\mathbb{I}}(L^\mathrm{SU}_{S,\mathbb{I}})=L^\mathrm{SU}_{S,f}$. 
\end{remark}

\section{Chida--Hsieh $p$-adic $L$-function} \label{CH}

Let 
$f=\sum_{n\geq 1}a_n(f)q^n$ be a modular 
form of level $\Gamma_0(Np)$ and even weight $k$, with trivial character satisfying 
Assumption \ref{ass}; since Assumption \ref{ass} is formally only stated for our fixed form $g$, 
we make the assumptions on $f$ (which could be different from $g$, even if later on will appear in the applications as arithmetic specialisation of the Hida family $\mathbf{f}$) more precise. We assume that
\begin{itemize}
\item $f\in S_k(\Gamma_0(Np))$ has level $\Gamma_0(Np)$, even weight $k\geq 2$ and trivial character; 
\item $p\nmid N$;  
\item $f$ is ordinary at $p$;
\item $f$ is a $p$-stabilised newform and we write 
$f_0$ for the unique newform of level $\Gamma_0(N)$ whose ordinary $p$-stabilisation is $f$ if $f$ is old at $p$, and $f_0=f$ if $f$ is a newform of level $\Gamma_0(Np)$. 
\end{itemize}
We keep the convention in $\S$\ref{subsec:main-result}, so in particular recall that $K$ is a quadratic imaginary extension of discriminant $D$ prime to $Np$, $p$ is split in $K$, the factorisation $N=N^+N^-$ is defined so that a prime number $\ell$ divides $N^+$ (respectively, $N^-$) if and only if it is split (respectively, inert) in $K$, and $N^-$ is a square-free product of an odd number of distinct primes. 

In \cite[Theorem 4.6]{ChHs1}, Chida--Hsieh construct an element 
$\Theta_p=\Theta_p(f)$ in the Iwasawa algebra  
$\mathcal{O}_L\pwseries{\Gamma_K^-}$ satisfying 
an interpolation property which we now describe.

\subsection{Setting the stage}
Let $\chi:\mathbb{A}_K^\times/K^\times\rightarrow \mathbb{C}^\times$ be an anticyclotomic algebraic Hecke character of infinity type $(m,-m)$ with \[-(k/2-1)\leq m\leq (k/2-1).\] 

Write $p = \mathfrak{p}  \cdot \overline{ \mathfrak{p}}$ in $K$ and assume that $\mathfrak{p}$ is compatible with  the chosen embedding $\iota_p$. 
Recall that $\mathbb{A}_{K,f}^\times$ is the subgroup of finite ideles of $\mathbb{A}_K^\times$.
By using the fixed embeddings $\iota_\infty$ and $\iota_p$, recall that the \emph{$p$-adic avatar} $\hat\chi:{\mathbb{A}_{K,f}^\times}/K^\times\rightarrow\bar\Q_p^\times$ of $\chi$ is defined by the locally algebraic character 
 \[\hat{\chi}(a)=\iota_p \iota^{-1}_\infty (\chi(\iota_{\mathrm{fin}}(a)))\cdot {( a_{\mathfrak{p}}  / a_{\overline{ \mathfrak{p}}} )^{m}},\]
where $\iota_{\mathrm{fin}}:\mathbb{A}_{K,f}^\times\hookrightarrow \mathbb{A}_K^\times$ is the map sending $x\mapsto (x,1)\in\mathbb{A}_{K,f}^\times\times \C^\times$.

Assume that the conductor of $\chi$ is $p^n$. Using the reciprocity map, $\hat\chi$ gives rise to a Galois character, denoted by the same symbol,  $\hat\chi:\Gamma_K^-\rightarrow\bar\Q_p^\times$ (as before, $\Gamma_K^-$ is the anticyclotomic $\Z_p$-extension of $K$). 
Let $\alpha_p=\alpha_p(f_0)$ be the unit root of the Hecke polynomial of $f_0$ at $p$ if $f$ is old at $p$, and $\alpha_p = a_p(f)$ if $f$ is new at $p$. 
Define 
\[{e_p(f,\chi)=\begin{cases} 1,\text{ if $n\geq 1$} \\ \left(1-\dfrac{\chi(\p)\cdot p^\frac{k-2}{2}}{\alpha_p}\right)\cdot\left(1-\dfrac{\bar\chi(\p)\cdot p^\frac{k-2}{2}}{\alpha_p}\right), \text{ if $n=0$.}
\end{cases}}\]
Write $u_K=\#\mathcal{O}_K^\times/2$. 

\subsection{Quaternionic modular forms and Gross periods}\label{gross-period}
Write $N^+=\mathfrak{N}^+\cdot\bar{\mathfrak{N}}^+$ as a factorisation of coprime ideals in $\mathcal{O}_K$. 

Let $B$ be the definite quaternion algebra over $\mathbb{Q}$ of discriminant $N^-$ and $R$ an Eichler order of level $N^+$.
Let $\phi_{f_0}$ be an integrally normalized Jacquet--Langlands transfer of a newform $f_0 \in S_k(\Gamma_0(N))$, i.e. a non-constant continuous function
$$\phi_{f_0} : B^{\times} \backslash \widehat{B}^{\times} / \widehat{R}^{(p), \times} \to \mathrm{Sym}^{k-2}(\mathcal{O}^2_L)$$
such that 
$\phi_{f_0} (a \cdot g \cdot r) =  r^{-1} \circ \phi_{f_0} (g)$ for $a \in B^{\times}$ and $r \in R^\times_p$, and
 the Hecke eigenvalues of $f_0$ and $\phi_{f_0}$ are the same at all primes not dividing $N^-$.
The integral normalization of $\phi_{f_0}$ is determined by the mod $p$ non-vanishing of the values of $\phi_{f_0}$ at the representatives of finite set $B^{\times} \backslash \widehat{B}^{\times} / \widehat{R}^{\times}\widehat{\mathbb{Q}}^\times$ in $\widehat{B}^{\times}$.  Here, we used fixed embeddings $\iota_p$ and $\iota_\infty$.
The space of such functions is denoted by $S^{N^-}_k(N^+, \mathcal{O}_L)$.
Recall that there is a pairing 
\[\langle - , - \rangle_k : \mathrm{Sym}^{k-2} (\mathcal{O}) \times \mathrm{Sym}^{k-2} (\mathcal{O}) \longrightarrow ((k-2)!)^{-1}\mathcal{O}\] defined in \cite[$\S$2.3]{ChHs1}, and define the pairing 
\begin{equation} \label{eqn:quaternionic_pairing}
\langle -, - \rangle_{N^+} : S^{N^-}_k(N^+, \mathcal{O}_L) \times S^{N^-}_k(N^+, \mathcal{O}_L) \to ((k-2)!)^{-1}\mathcal{O}_L
\end{equation}
as in \cite[(6.1)]{ChHs1} by the formula, for $\psi_1, \psi_2\in  S^{N^-}_k(N^+, \mathcal{O}_L)$, 
$$\langle\psi_1, \psi_2 \rangle_{N^+ } := \sum_{[b]} 
\dfrac{1}{  \# \left((B^\times \cap b\widehat{R}^\times b^{-1}\widehat{\mathbb{Q}}^\times ) / \mathbb{Q}^\times\right)  } \cdot
\langle \psi_1(b), \psi_2(bw_{N^+}) \rangle_k $$
where
$w_{N^+}$ is the Atkin--Lehner operator for level $N^+$ (\emph{cf.} \cite[\S3.3]{ChHs1}),
$[b]$ runs over a set of representatives of $B^\times \backslash \widehat{B}^\times / \widehat{R}^\times\widehat{\mathbb{Q}}^\times$. 

Let $\xi_{\phi_{f_0}}(N^+, N^-) = \langle \phi_{f_0}, \phi_{f_0} \rangle_{N^+}$ be the quaternionic analogue of the cohomology congruence ideal for $\phi_{f_0}$ using the above pairing (\ref{eqn:quaternionic_pairing}) as in \cite[$\S$2.1]{pollack-weston} and \cite[(3.9) and (4.3)]{ChHs1}.
Define the \emph{Gross period} by
\[\Omega_{f_0,N^-}=\dfrac{(4\pi)^k\langle f_0,f_0\rangle_{N}}{ \xi_{\phi_{f_0}}(N^+, N^-) }\]
where 
\[\langle f_0,f_0\rangle_{N} = \int \int_{\Gamma_0(N) \backslash \mathfrak{h}} f_0(z) \cdot \overline{f_0(z)}y^{k-2} dxdy\] is the Petersson norm of $f_0$; 
see \cite[(4.3), Remark (ii) and \S6]{ChHs1} for details. 

\begin{remark} \label{rem:p-stabilization}
We can also repeat the process above for a $p$-stabilized newform $f$.
When $f$ is the $p$-stabilization of $f_0$, it is not difficult to see that $\Omega_{f,N^-} = \Omega_{f_0,N^-}$ up to a $p$-adic unit with help of Ihara's lemma for quaternion algebras \cite[Theorem 5.13]{kim-asian} under Assumption \ref{ass}.(1). See \cite[Lemma 3.6]{pw-mt} for the details.
\end{remark}
\subsection{The interpolation formula}
When $f$ is the $p$-stabilization of a newform $f_0$ of level $\Gamma_0(N)$, let $\epsilon_p(f)=1$.
When $f$ is new at $p$, let $\epsilon_p(f)\in\{\pm 1\}$ is the eigenvalue of the Atkin--Lehner involution at $p$ acting on $f$; in particular, $\epsilon_p(f)=-{p^{-\frac{k-2}{2}}}\cdot{a_p(f)}$. 
Define \[u(f,p)=u_K^2\cdot\sqrt{D}\cdot\chi(\mathfrak{N}^+)\cdot D^{k-2}
\cdot\epsilon_p(f)\cdot (-1)^m.\] 
Finally, let $\Gamma(s)$ be the complex $\Gamma$-function; recall that 
$\Gamma(j)=(j-1)!$ if $j\geq 1$ is an integer.  
Then we have
\begin{equation} 
\label{eq-Ch-Hs}
\hat\chi(\Theta_p^2)=e_p(f,\chi)^{2-t}\cdot u(f,p) \cdot
\Gamma\left(\frac{k}{2}+m\right)\cdot\Gamma\left(\frac{k}{2}-m\right)\cdot 
\frac{p^{n(k-1)}}{\alpha_p^{2n}}\cdot
\frac{L_K(f_0,\chi,k/2)}{\Omega_{f,N^-}},\end{equation}
where recall that $f_0 = f$ when $f$ is new at $p$, in which case we put $t=1$,  and $f$ is the ordinary $p$-stabilisation of 
$f_0$ when $f$ is old at $p$, in which case we put $t=0$. 

\begin{definition}\label{def-Ch-Hs}
Define the \emph{Chida--Hsieh} $p$-adic $L$-function of $f$ 
to be $L^\mathrm{CH}_f=\Theta_p^2(f)$ in $\mathcal{O}\pwseries{\Gamma_K^-}$.   
\end{definition}

\begin{remark}
In \cite[\S4]{ChHs1} several theta elements $\Theta_n^{[m]}$ for $- k/2<m<k/2$ and $n\geq 1$ an integer are considered. A priori,
 $\Theta_n^{[m]}(f)$ belongs to $((k-2)!)^{-1}\mathcal{O}\pwseries{\mathscr{G}_n}$ for $\mathscr{G}_n\simeq \Z/p^n\Z$ the quotient of $\Gamma_K^-$ of order $p^n$; however, if $m=k/2-1$, then $\Theta^{[k/2-1]}(f)$ is integral (see also \cite[Remark 2.5]{cas-longo}), and from  \cite[Corollary 4.5]{ChHs1} it follows that all of these elements $\Theta_n^{[m]}$ are integral, so in fact $L^\mathrm{CH}_f$ belongs to $\mathcal{O}\pwseries{\Gamma_K^-}$. 
\end{remark}

\begin{remark}
If $k=2$, the \emph{square-root} $p$-adic $L$-function $\Theta_p$ has a long story and has been studied, among others, by Hida \cite{Hida1}, \cite{Hida2}, Perrin-Riou \cite{PR}, Vatsal \cite{Vat1}, and Bertolini--Darmon \cite{BD-Heegner}, \cite{BDmumford-tate}, \cite{bdIMC} (with slightly different interpolation formulas); however, we decided the name of $L^\mathrm{CH}_f$ because we mainly follow the presentation of this $p$-adic analytic function settled by Chida--Hsieh in \cite{ChHs1}. 
\end{remark}

\section{Periods and congruence ideals} 
In this section we study the relation between periods and congruence ideals of modular forms. In this section $N$ is an integer, $p\nmid N$ a prime number, 
$K$ a quadratic imaginary field of discriminant prime to $Np$, $N=N^+N^-$ the factorisation of $N$ where a prime number $\ell\mid N^+$ if and only if it is split in $K$. We assume as before that $N^-$ is a square-free product of an odd number of primes. 

Let $S_k(\Gamma_0(N), \mathcal{O}_L)$ be the $\mathcal{O}_L$-module of weight $k$ modular forms on $\Gamma_0(N)$ with coefficients in $\mathcal{O}_L$ with respect to the chosen embedding $\iota_p$. 
Let $\mathfrak{m}$ be a non-Eisenstein maximal ideal of the full Hecke algebra over $\mathcal{O}_L$ acting faithfully on $S_k(\Gamma_0(N), \mathcal{O}_L)$.
Let $\mathbb{T}_{N}$ denote the localization of the full Hecke algebra at $\mathfrak{m}$.
In other words, 
$\mathbb{T}_{N}$ is the $\mathcal{O}_L$-subalgebra of $\End_{\mathcal{O}_L}(S_k(\Gamma_0(N),\mathcal{O}_L)_{\mathfrak{m}} )$  generated by Hecke operators $T_\ell$ for primes $\ell\nmid N$ and $U_\ell$ for primes $\ell\mid N$.
We also denote $\mathbb{T}_{N}^{\new}$ the quotient of $\mathbb{T}_{N}$ acting faithfully on the submodule of $S_k(\Gamma_0(N),\mathcal{O}_L)_{\mathfrak{m}}$ consisting of forms which are new at all primes dividing $N^-$.
 
Fix an eigenform $f_0 \in S_k(\Gamma_0(N),\mathcal{O}_L)_{\mathfrak{m}}$ which is new at all primes dividing $N^-$. 
Let \[\theta_{f_0}:\mathbb{T}_{N} \longrightarrow \mathcal{O}_L\] be the morphism associated with $f_0$.  
Let $\eta_{f_0}$ be a generator of the $\mathcal{O}_L$-ideal 
\[\theta_{f_0}\left(\mathrm{Ann}_{\mathbb{T}_{N}}(\ker(\theta_{f_0}))\right).\] The ideal $(\eta_{f_0})$ is called the \emph{congruence ideal} of $f_0$.
 Since $f_0$ is new at all primes dividing $N^-$, $\theta_{f_0}$ factors through the canonical projection $\mathbb{T}_{N}\twoheadrightarrow\mathbb{T}_{N}^\new$, and we obtain  a morphism 
$\theta^{\new}_{f_0}:\mathbb{T}_{Np}^{\new}\rightarrow\mathcal{O}_L$. Denote by
$\eta_{f_0,N^-}$ a generator of the $\mathcal{O}_L$-ideal
\[\theta_{f_0}^\new(\mathrm{Ann}_{\mathbb{T}_{N}^\new}(\ker(\theta_{f_0}^\new))).\]
See 
\cite[\S2.2]{pollack-weston} (when $k=2$) and \cite[\S6]{ChHs1} (for $k\geq 2$) 
for details. Recall the canonical periods $\Omega_f^+$, $\Omega_f^-$ defined in \S\ref{sec4.2} (these are well defined up to $p$-adic units) and the Gross period $\Omega_{f_0,N^-}$ introduced in \S\ref{gross-period} (under Assumption \ref{ass}). 

\begin{proposition}\label{congruence} 
Under Assumption $\ref{ass}$, up to $p$-adic units, we have 
\begin{enumerate}
\item $\displaystyle{\Omega_{f_0,N^-}=\frac{ (4\pi)^k \cdot \langle f_0,f_0\rangle_{N}}{\eta_{f_0,N^-}}}$.
\item $\displaystyle{\Omega_{f_0}^+\cdot\Omega_{f_0}^-=\frac{ \langle f_0,f_0\rangle_{N}}{\eta_{f_0}}}$.
\end{enumerate}
\end{proposition}
\begin{proof}
The first statement follows from \cite[Proposition 6.4(2)]{pollack-weston} if $k=2$. For $k\geq 2$, this is 
\cite[Corollary 6.7]{KO} or \cite[Remark 7.8]{Hsieh}.
The second statement follows from \cite[Lemma 12.1]{SU}. 
\end{proof}
The same result holds for $p$-ordinary $p$-stabilized newforms of level $\Gamma_0(Np)$. 
See also Remark \ref{rem:p-stabilization}.

\section{Comparison between Skinner--Urban and Chida--Hsieh: weight two forms}\label{SU-CH-section} 

In this section we compare the one-variable anticyclotomic specialisation of the three-variable $p$-adic $L$-function of Skinner--Urban with the $p$-adic $L$-function of Chida--Hsieh, in the case of modular forms of weight $2$. Although this result is well-known to the experts, and already stated in the case of elliptic curves (see \cite[\S3.6.3]{SU} but with no proof), we add a proof which might serve as a reference. 

Let $f=\mathbf{f}_\phi\in S_2(\Gamma_0(Np))$ as in \eqref{f_phi} the specialisation of a Hida family $\mathbf{f}$ at an arithmetic character $\phi$, with trivial character $\chi_\phi$ and weight $2$.
Then $f$ is either a newform of level $\Gamma_0(Np)$, in which case we set $f_0=f$, or $f$ is the ordinary $p$-stabilisation of an ordinary newform $f_0$ of level $\Gamma_0(N)$; let as before $\alpha_p=\alpha_p(f_0)$ be the unit root of the Hecke polynomial of $f_0$ at $p$ if $f$ is old at $p$, and $\alpha_p = a_p(f)$ if $f$ is new at $p$, where $f=\sum_{n\geq 1}a_n(f)q^n$. 
 
Let $K$ be an imaginary quadratic field of discriminant prime to $Np$ in which $p$ is split, inducing the factorisation $N=N^+N^-$ as before. Let $\chi:\mathbb{A}_K^\times/K^\times\rightarrow\C^\times$ be a finite order anticyclotomic algeraic character and denote $\hat\chi$ be its $p$-adic avatar.

For any prime $\ell\mid S$, we denote $L_\ell(f_0,\chi,s)$ the local Euler factor at $\ell$ of the complex $L$-function $L_K(f_0,\chi,s)$, 
so that we have  $L_K(f,\chi,s)=\prod_\ell L_\ell(f_0,\chi,s)$, where the product is over all prime numbers $\ell$. 
The description of the Euler factors $L_\ell(f,\chi,s)$, following \S\ref{Euler factors}, is the following.
For each prime $v$ of $K$, let $\ell$ be the rational prime lying below $v$  and let $\alpha_\ell(f)$ and $\beta_\ell(f)$ the roots of the Hecke polynomial of $f$ at $\ell$; here, with a slight abuse of notation, we understand that some of these roots may be zero. 
Define 
\[L_v(f,\chi,s)=(1-\chi(v)\alpha_{\N(v)}(f)\N(v)^{-s})^{-1}\cdot (1-\chi(v)\beta_{\N(v)}(f)\N(v)^{-s})^{-1}.
\] 
Then for each rational prime $\ell$, we have 
\[L_\ell(f,\chi,s)=\prod_{v\mid\ell}L_v(f,\chi,s).\] 
We also have the following alternative description of these factors. Let
\[\mathit{\Psi}_ K ^- :\mathbb A_ K ^\times/ K ^\times
\overset{\mathrm{rec}_ K }\longrightarrow \Gal( K ^\mathrm{ab}/ K )\overset{\boldsymbol{\varepsilon}_ K^-}
\longrightarrow (\Lambda_ K ^-)^\times\]
be the map obtained from the reciprocity map and the canonical projections. Then $\hat\chi$ extends to a homomorphism, denoted with the same symbol, $\hat\chi:\Lambda_K^-\rightarrow\bar\Q_p$. If $\p_v$ is a uniformizer element of $v$, then  
\[L_v(f,\chi,s)=(1-\hat\chi(\mathit{\Psi}_ K ^-(\pi_v)) \cdot \alpha_{\N(v)}(f)\cdot\N(v)^{-s})^{-1}\cdot (1-\hat\chi(\mathit{\Psi}_ K ^-(\pi_v))\cdot\beta_{\N(v)}(f) \cdot \N(v)^{-s})^{-1}.\]
Define for each prime ideal $v\mid S$ of $K$ the element $\mathcal{E}_v\in \mathcal{O}\pwseries{\Gamma_K^-}$ by  
\[\mathcal{E}_v=(1-\mathit{\Psi}_ K ^-(\p_v)\cdot\alpha_{\N(v)}(f) \cdot \N(v)^{-{1}})\cdot (1-\mathit{\Psi}_ K ^-(\p_v)\cdot\beta_{\N(v)}(f) \cdot \N(v)^{-{1}}).\]
Set $\mathcal{E}_\ell=\prod_{v\mid\ell}\mathcal{E}_v$ and 
$\mathcal{E}_S=\prod_{\ell\mid S}\mathcal{E}_\ell$. In particular, $\hat\chi(\mathcal{E}_\ell^{-1})=L_\ell(f,\chi,1)$, and therefore 
\begin{equation}\label{completed-incompleted1}
L^S_K(f,\chi,1)=\hat\chi(\mathcal{E}_S)\cdot L_K(f,\chi,1).\end{equation}
Moreover, since $S$ contains $p$, we have 
\begin{equation}\label{fg}
L^S_K(f,\chi,1)=L^S_K(f_0,\chi,1).\end{equation}
Define $L^\mathrm{CH}_{S,f}=\mathcal{E}_S\cdot\Theta_p^2(f)$.

\begin{theorem} \label{SU-CH} 
Let $f$ be the $p$-stabilization of a $p$-ordinary newform $f_0$ of level $\Gamma_0(N)$. Then under Assumption $\ref{ass}$ we have 
$$(L^\mathrm{SU}_{S,f})=(L^\mathrm{CH}_{S,f}) \in \mathcal{O}\pwseries{\Gamma_K^-}.$$
\end{theorem}
\begin{proof}We show that the powers series $L^\mathrm{SU}_{S,f}$ and $L^\mathrm{CH}_{S,f}$ agree when evaluated at infinitely many characters $\chi$, up to a $p$-adic unit which is independent of $\chi$, and then the result follows from the Weierstrass preparation theorem. Fix a finite order character $\hat\chi:\Gamma_K^-\rightarrow\bar\Q_p$ of conductor $p^n$ for some integer $n>t_\phi $, which is the $p$-adic avatar of a finite order anticyclotomic algebraic Hecke character $\chi:\mathbb{A}_K^\times/K^\times\rightarrow\C^\times$ of conductor $p^n$ (so
the infinity type of $\chi$ is $(0,0)$, and $\chi$ and $\hat\chi$ are simply related by the geometrically normalised reciprocity map). 

Choose $\phi$ such that its restriction to the anticyclotomic line $\Gamma_K^-$ is $\hat\chi$. Define  
\[u_\mathrm{SU}(\hat\chi)=u_f\cdot{\mathfrak{g}(\theta_\phi)}/{p^n}.\]
By Lemma \ref{gauss}, 
$\mathfrak{g}(\theta_\phi)=\pm p^n$, so $u_\mathrm{SU}(\hat\chi)=\pm u_f$ 
and therefore $u_\mathrm{SU}(\hat\chi)$ 
is a $p$-adic unit, which depends on $\chi$ only up to a sign. We therefore have: 
\[L^\mathrm{SU}_{S,f}(\hat\chi)=
\frac{u_\mathrm{SU}(\hat\chi)\cdot p^n}{\alpha_p^{2n}}
\cdot \frac{L^S_K(f_0,\chi,1)}{\pi^2\cdot\Omega_{f_0}^+\cdot\Omega_{f_0}^-} .\]
On the other hand, by Equation \eqref{eq-Ch-Hs} and Definition \ref{def-Ch-Hs}, we have  
\[
L_{S,f}^\mathrm{CH}(\hat\chi)=
\frac{\chi(\mathfrak{N}^+)\cdot u_\mathrm{CH}\cdot p^{n}}{\alpha_p^{2n}}\cdot
\frac{L^S_K(f,\chi,1)}{  \Omega_{f_0,N^-}},\]
where $u_\mathrm{CH}$ is a unit which does not depend on $\chi$. 
We thus have 
\[{L_{S,f}^\mathrm{SU}}(\hat\chi)=\frac{u_\mathrm{SU}(\chi)}{\chi(\mathfrak{N}^+)\cdot u_\mathrm{CH}}\cdot
\frac{\Omega_{f_0,N^-}}{ \pi^2 \cdot \Omega_{f_0}^+\cdot\Omega_{f_0}^-}\cdot {L_{S,f}^\mathrm{CH}}(\hat\chi).\]
 Now $u_\mathrm{CH}$ is a $p$-adic unit independent of $\chi$.  
Thanks to Remark \ref{rem:p-stabilization}, $\Omega_f^\pm=\Omega_{f_0}^\pm$, namely, the periods of $f$ and $f_0$ differ by $p$-adic units. 
By Proposition \ref{congruence}, the 
quotient between the Gross periods $\Omega_{f_0,N^-}$ and 
$(4 \pi)^2\cdot  \Omega_{f_0}^+\cdot\Omega_{f_0}^-$ is equal to $\eta_{f_0}/\eta_{f_0,N^-}$, which is a $p$-adic unit under Assumption \ref{ass} by \cite[Lemma 9.2]{SZ} with $p >2$. 
Moreover, we have infinitely many Hecke characters $\chi$ as above such that 
$u_\mathrm{SU}(\chi)=\epsilon$ for at least one choice of $\epsilon\in\{\pm1\}$, which we fix in the following considerations.  
Therefore $u={u_\mathrm{SU}(\chi)}/{u_\mathrm{CH}}$ 
is a unit independent of $\chi$, for $\chi$ in an infinite set of characters. Define 
$\alpha=\mathrm{rec}_K(\mathfrak{N}^+)$, so that 
$\hat{\chi}(\alpha)=\chi(\mathfrak{N}^+)$. 
The values $u  \cdot {L_{S,f}^\mathrm{SU}}(\hat\chi)$ and $\hat\chi(\alpha) \cdot  \eta_{f,N^-} \cdot {L_{S,f}^\mathrm{CH}}(\hat\chi)$ are equal for infinitely many anticyclotomic Hecke characters $\chi$, and therefore, 
using the Weierstrass preparation theorem, we see that 
$u \cdot  {L_S^\mathrm{SU}}$ and $\alpha \cdot  \eta_{f,N^-} \cdot  {L_{S,f}^\mathrm{CH}}$ are equal in $\mathcal{O}\pwseries{\Gamma_K^-}$. Now $\alpha$, $\eta_{f,N^-}$ and $u$ are units, completing the proof.\end{proof}

\section{Hida--B\"{u}y\"{u}kboduk--Lei $p$-adic $L$-function}
We recall a variant of Hida--Perrin-Riou $p$-adic $L$-function in \cite{Hida1}, \cite{PR}, recently developed by \cite{LLZ}, \cite{BL1}, \cite{BL2}. We will mainly follow the presentation of \cite[\S2]{BL1}, \cite[Appendix B]{BL2}. 

Let $f_0$ be a newform of level $\Gamma_0(N)$ with $p\nmid N$ and even weight $k\geq 2$. 
Let $\alpha_p$ and $\beta_p$ are the roots of the Hecke polynomial of $f_0$ at $p$.
Denote by $\alpha_p$ the unit root and by $f$ the $p$-stabilisation of $f_0$ with $U_p$-eigenvalue $\alpha_p$, as before.

Let $\chi$ be an algebraic Hecke character of $K$ of infinity type $(t_1,t_2)$ and denote 
\[\hat\chi:{\mathbb{A}_{K,f}^\times}/K^\times\longrightarrow\bar\Q_p^\times\] the $p$-adic avatar of $\chi$; using the same conventions as in Section \ref{CH}, $\hat\chi$ is identified with a Galois character, denoted with the same symbol, 
$\hat\chi:G_K^\mathrm{ab}\rightarrow\bar\Q_p^\times$ by composing with the (geometrically normalised) reciprocity map. 

Denote by $\Sigma^{(1)}$ the set of Hecke characters of infinity type $(t_1,t_2)$ 
with \[-(k/2-1)\leq t_1,t_2\leq k/2-1.\] 
Fix an integral ideal $\mathfrak{f}$ prime to $p$ and denote $H_{\mathfrak{f}p^\infty}$ the ray class group of conductor $\mathfrak{f}p^\infty$, by which we mean the inverse limit of all ray class groups of conductors $\mathfrak{f}\p^m\bar{\p}^n$ over all non-negative integers $m$ and $n$; then we have a projection map $G_K^\mathrm{ab}\rightarrow H_{\mathfrak{f}p^\infty}$. 

Consider the family of $p$-depleted theta series, following  \cite[\S 6.2]{LLZ}, 
\cite[\S 2.2]{BL1}, \cite[Appendix B]{BL2}: 
\[\Theta=\sum_{(\mathfrak{a},p)=1}[\mathfrak{a}]q^{\N(\mathfrak{a})}\]
where the sum is over all ideals $\mathfrak{a}$ of $K$ which are coprime to $p$ and $[\mathfrak{a}]$ is the class of $\mathfrak{a}$ in $H_{\mathfrak{f}p^\infty}$; so $\Theta$ is an element of $\mathcal{O}\pwseries{H_{\mathfrak{f}p^\infty}}\pwseries{q}$, where $\mathcal{O}\pwseries{H_{\mathfrak{f}p^\infty}}$  is the Iwasawa algebra of the $p$-adic Lie group $H_{\mathfrak{f}p^\infty}$ with coefficients in $\mathcal{O}$, where $\mathcal{O}$ is, as before, the valuation ring of a finite extension of $\Q_p$. For any Hecke character $\chi$ of conductor $\mathfrak{f}p^\infty$, we then have 
\[\Theta(\hat\chi)=
\hat\chi(\Theta)=\sum_{(\mathfrak{a},p)=1}
\hat\chi(\mathfrak{a})q^{\N(\mathfrak{a})}.\] 

For any Hecke character $\chi$ of prime to $p$-conductor $\mathfrak{f}$, such that $\chi\N^{j-k/2}$ belongs to $\Sigma^{(1)}$ (so, $1\leq j\leq k-1$), we introduce the following quantities: 
\begin{itemize}
\item For each integer $j$,
 \[\mathcal{E}(f,\chi,j)=\prod_{\mathfrak{q}\in\{\p,\bar\p\}}\left(1-\frac{p^{j-1}}{\alpha_p\chi(\mathfrak{q})}\right)\left(1-\frac{\beta_p\chi(\mathfrak{q})}{p^j}\right);\] 
\item $\mathcal{E}(f)=1-\frac{\beta_p}{p\alpha_p}$;
\item $\mathcal{E}^*(f)=1-\frac{\beta_p}{\alpha_p}$. 
\end{itemize}
By \cite[Theorem 2.1 in Appendix B]{BL2} there exists a $p$-adic $L$-function 
\[L(f/K,\Sigma^{(1)})\in 
\mathcal{O}\pwseries{H_{\mathfrak{f}p^\infty}}\otimes_{\mathcal{O}}L,\] 
where $L$ is a sufficiently big finite extension of $\Q_p$ with valuation ring $\mathcal{O}$, satisfying the following interpolation formula. 
If $\chi$ is a finite idele class character, and $j$ is an integer such that $1\leq j\leq k-1$, then:
\begin{itemize}
\item If the conductor of $\chi$ is prime to $p$, we have 
\[L(f/K,\Sigma^{(1)})(\hat\chi\cdot\N^{j-k/2})=\frac{\mathcal{E}(f,\chi,j)\cdot 
u\cdot \Gamma(j)^2}{\mathcal{E}(f)\cdot\mathcal{E}^*(f)}\cdot
\frac{L_K(f_0,\chi,j)}{\pi^{2j}\cdot\langle f_0,f_0\rangle_N}, 
\]
where $u$ is a unit independent of $\chi$; 
\item  If the $p$-primary part of the conductor of $\chi$ is $p^n$ with $n\geq 1$, we have 
\[L(f/K,\Sigma^{(1)})(\hat\chi\cdot\N^{j-k/2})=
\frac{p^{2jn}\cdot\tau(\hat\chi)\cdot u\cdot \Gamma(j)^2}{\alpha_p^{2n}\cdot\mathcal{E}(f)\cdot\mathcal{E}^*(f)}\cdot
\frac{L_K(f_0,\bar\chi,j)}{\pi^{2j}\cdot\langle f_0,f_0\rangle_N},\]
where $u$ is a unit independent of $\chi$ and $\tau(\hat\chi)$ is 
the root number of $\Theta(\hat\chi)$. 
\end{itemize}

We can view $L(f/K,\Sigma^{(1)})$ as an element of 
{$\mathcal{O}\pwseries{\Gamma_K^-}\otimes_{\mathcal{O}}L$} by composing with the canonical 
projection $\pi:H_{\mathfrak{f}p^\infty}\rightarrow \Gamma_K^-$, and the composition is independent of the chosen $\mathfrak{f}$ as long as $p$ does not divide the cardinality of the ray class group of conductor $\mathfrak{f}$. Then define for any $\mathcal{O}$-algebra homomorphism $\chi:\mathcal{O}\pwseries{\Gamma_K^-}\rightarrow\bar\Q_p$, 
\[L^\mathrm{HBL}_{f,j}(\chi)=L(f/K,\Sigma^{(1)})((\chi\circ\iota\circ\pi)\cdot\N^{j-k/2})\] where $\pi: \mathcal{O}\pwseries{H_{\mathfrak{f}p^\infty}}\rightarrow \mathcal{O}\pwseries{\Gamma_K^-}$ is the canonical map and $\iota:\mathcal{O}\pwseries{\Gamma_K^-}\rightarrow\mathcal{O}\pwseries{\Gamma_K^-}$ is the $\mathcal{O}$-algebra map induced from 
$\gamma\mapsto\gamma^{-1}$ for $\gamma\in\Gamma_K^-$. In particular, 
if $\chi$ is a finite idele class character such that the $p$-primary part of its conductor is $p^n$ with $n\geq 1$, and $j$ is an integer such that $1\leq j\leq k-1$, then
\[L^\mathrm{HBL}_{f,j}(\hat\chi)=
\frac{p^{2jn}\cdot\tau(\hat\chi)\cdot u\cdot \Gamma(j)^2}{\alpha_p^{2n}\cdot\mathcal{E}(f)\cdot\mathcal{E}^*(f)}\cdot
\frac{L_K(f_0,\chi,j)}{\pi^{2j}\cdot\langle f_0,f_0\rangle_N},\]
where, as before, $u$ denotes a unit independent of $\chi$ and $\tau(\hat\chi)$ is 
the root number of $\Theta(\hat\chi)$ (note that, up to $p$-adic units independent of $\chi$, the root number of the Theta series associated with $\chi$ is the same as the root number of the Theta series associated with $\bar{\chi}$ by the explicit formulas in \S\ref{sectheta}). 

\section{Comparison between Hida--B\"{u}y\"{u}kboduk--Lei and Chida--Hsieh}
\label{Sec-HBL-CH} 
As in \S\ref{SU-CH-section}, 
let $f=\mathbf{f}_\phi\in S_k(\Gamma_0(Np))$ be the ordinary $p$-stabilisation of a newform $f_0$ of level $\Gamma_0(N)$. We assume that $k\geq 4$ is even and $f$ has trivial character. Let $K$ be an imaginary quadratic field of discriminant prime to $Np$ as before, so $p$ is split in $K$ and the decomposition $N=N^+N^-$ satisfies the condition that $N^-$ is a square-free product of an odd number of primes. 

\begin{theorem}\label{Hida-CH} Under Assumption $\ref{ass}$, if $k\geq 4$ we have  
$(\eta_{f,N^-}\cdot L_{f,k/2}^\mathrm{HBL})=(L_f^\mathrm{CH})$ as ideals of $\mathcal{O}\pwseries{\Gamma_K^-}$. 
\end{theorem}
\begin{proof}
As in the proof of Theorem \ref{SU-CH}, we show the equality of the two $p$-adic $L$-functions when evaluated at infinitely many characters $\chi$, up to a $p$-adic unit which is independent of $\chi$, and then the result follows from the Weierstrass preparation theorem. Fix a finite order character $\hat\chi:\Gamma_K^-\rightarrow\bar\Q_p^\times$ of conductor $p^n$ for some integer $n\geq 1$, which is the $p$-adic avatar of a finite order anticyclotomic algebraic Hecke character $\chi:\mathbb{A}_K^\times/K^\times\rightarrow\C^\times$ of conductor $p^n$; thus as before the infinity type of $\chi$ is $(0,0)$ and 
$\hat\chi$ and $\chi$ are related by the reciprocity map.   
For $j=k/2$ in the formula for $L(f/K,\Sigma^{(1)})(\hat\chi)$ we obtain 
\[L^\mathrm{HBL}_{f,k/2}(\hat\chi)=
\frac{p^{kn}\cdot\tau(\hat\chi)\cdot u\cdot \Gamma(k/2)^2}{\alpha_p^{2n}\cdot\mathcal{E}(f)\cdot\mathcal{E}^*(f)}\cdot
\frac{L_K(f_0,\chi,k/2)}{\pi^{k}\cdot\langle f_0,f_0\rangle_N}.\]
Here recall that $u$ is a unit independent of $\chi$. 
Since $k\geq 4$, both $\mathcal{E}(f)$ and $\mathcal{E}^*(f)$ are both $p$-adic units, independent of $\chi$, so we may write 
\[L^\mathrm{HBL}_{f,k/2}(\hat\chi)=
\frac{p^{kn}\cdot\tau(\hat\chi)\cdot u_{\mathrm{HBL}}\cdot \Gamma(k/2)^2}{\alpha_p^{2n}}\cdot
\frac{L_K(f_0,\chi,k/2)}{\pi^{k}\cdot\langle f_0,f_0\rangle_N}\] for a $p$-adic unit $u_\mathrm{HBL}$ which is independent of $\chi$.

On the other hand, by Equation \eqref{eq-Ch-Hs} and Definition \ref{def-Ch-Hs}, 
and using that $e_p(f,\chi)=1$ if $n\geq 1$, we have 
\[
L^\mathrm{CH}_f(\hat{\chi})=\frac{u_{\mathrm{CH}}\cdot\bar\chi(\mathfrak{N})\cdot\Gamma(k/2)^2\cdot p^{n(k-1)}}{\alpha_p^{2n}} 
\cdot
\frac{L_K(f_0,\chi,k/2)}{ \Omega_{f,N^-}}\]
where $u_{\mathrm{CH}}$ is a $p$-adic unit, independent of $\chi$.

Comparing with the expression in Theorem \ref{SU-CH}, we obtain 
\begin{equation}\label{eqHBL1}
L^\mathrm{HBL}_{f,k/2}(\hat\chi)=p^{n}\cdot\tau(\hat\chi)
\cdot
\frac{\tilde{u}}{\bar{\chi}(\mathfrak{N}^+)}\cdot\frac{\Omega_{f_0,N^-}}{(4 \pi)^{k}\cdot \langle f_0,f_0\rangle_N}\cdot 
L^\mathrm{CH}_f(\hat{\chi}),\end{equation}
where 
$\tilde{u}$ is a $p$-adic unit independent of $\chi$. 

We have 
\[\eta_{f_0,N^-}=\frac{(4 \pi)^{k}\cdot\langle f_0,f_0\rangle_N}{\Omega_{f_0,N^-}}\]
 up to $p$-adic units 
by Proposition \ref{congruence}.
The explicit formulas in 
\cite[(5.5b)]{Hida2} (see also \S\ref{sectheta}) show that $\tau(\hat\chi)=p^{-n}$, up to a $p$-adic unit independent of $\chi$; note that the action \cite{BL1} and \cite{BL2} of a matrix $\gamma=\smallmat abcd\in\GL_2(\R)$ with positive determinant on a modular form $g$ is via the formula \[g_{|\ell}\gamma(z)=\det(\gamma)^{\ell-1}(cz+d)^{-\ell}g(\gamma(z)),\] which explains the discrepancy between the root numbers used in \cite{BL1}, \cite{BL2} and \cite{Hida1}, \cite{Hida2} since the last two references use the action introduced in \S\ref{sectheta}). 
 The conclusion follows then from Equation \ref{eqHBL1} in light of the observations made in this paragraph. 
\end{proof}

\section{Comparison between Skinner--Urban and Hida--B\"{u}y\"{u}kboduk--Lei} 

As in \S\ref{SU-CH-section} and \S\ref{Sec-HBL-CH}, 
let $f=\mathbf{f}_\phi\in S_k(\Gamma_0(Np))$ be the ordinary $p$-stabilisation of a newform $f_0$ of level $\Gamma_0(N)$ of even weight $k\geq 4$ and trivial character. Let $K$ be an imaginary quadratic field of discriminant prime to $Np$ as before, so $p$ is split in $K$ and the decomposition $N=N^+N^-$ satisfies the condition that $N^-$ is a square-free product of an odd number of primes. 
Similarly as before, and using the notation of \S\ref{SU-CH-section}, 
define 
\[L^\mathrm{HBL}_{S,f,j}=\mathcal{E}_S\cdot L^\mathrm{HBL}_{f,j}.\]

\begin{theorem}\label{HBL-SU}Under Assumption $\ref{ass}$, if $k\geq 4$ then 
\begin{enumerate}
\item $(\eta_f\cdot L^\mathrm{HBL}_{S,f,k/2})=( L^\mathrm{SU}_{S,f})$, 
\item 
$(\eta_f\cdot L^\mathrm{HBL}_{S,f,k-1})=(\pi_f^\mathrm{ac}(\mathcal L^{S}_{\mathbf{f},K }))$. 
\end{enumerate}

\end{theorem}
\begin{proof}
First observe that two formulas are equivalent under twisting by
$\mathrm{Tw}_{\mathbf{\Theta}^{-1}}(\gamma)$. 
The first formula is proved in \cite{BL1}, where one first shows in Remark 2.2 that the values on anticyclotomic characters $\chi$ of $L^\mathrm{SU}_{S,f}$ and  $L^\mathrm{HBL}_{S,f,k/2}$ differ by a $p$-adic unit, and then the proof of Theorem 3.20 shows that $L^\mathrm{SU}_{S,f}$ divides $L^\mathrm{HBL}_{S,f,k/2}$. An alternative proof of the second equivalent statement can be obtained by an explicit comparison of the interpolation formulas, as in the proof of Theorems \ref{Hida-CH} and 
\ref{SU-CH}. For the reader's convenience, and to fully explain the presence of the congruence ideal which was missing in \cite{BL1}, we offer a complete proof.

As in the proof of Theorems \ref{SU-CH} and \ref{Hida-CH}, we show the equality of the two $p$-adic $L$-functions when evaluated at infinitely many characters $\chi$, up to a $p$-adic unit which is independent of $\chi$, and then the result follows from the Weierstrass preparation theorem. Fix a finite order character $\hat\chi:\Gamma_K^-\rightarrow\bar\Q_p^\times$ of conductor $p^n$ for some integer $n\geq 1$, which is the $p$-adic avatar of a finite order anticyclotomic algebraic Hecke character $\chi:\mathbb{A}_K^\times/K^\times\rightarrow\C^\times$ of conductor $p^n$. 
Recall from the proof of Theorem \ref{Hida-CH} that  
$\tau(\hat\chi)=p^{-n}$, up to a $p$-adic unit independent of $\chi$, and 
that since $k\geq 4$ both $\mathcal{E}(f)$ and $\mathcal{E}^*(f)$ are $p$-adic units, independent of $\chi$. Thus for $j=k-1$ in the formula for $L(f/K,\Sigma^{(1)})(\hat\chi)$,
we obtain 
\[ L^\mathrm{HBL}_{S,f,k-1}(\hat\chi)=
\frac{p^{n(2k-3)}\cdot u_\mathrm{HBL}\cdot \Gamma(k-1)^2}{\alpha_p^{2n}}\cdot
\frac{L^S_K(f_0,\chi,k-1)}{\pi^{2k-2}\cdot\langle f_0,f_0\rangle_N},\]
where $u_\mathrm{HBL}$ is a $p$-adic unit independent of $\chi$. 
By Lemma \ref{gauss}, 
$\mathfrak{g}(\theta_\phi)=\pm p^n$, so, after setting 
\[u_\mathrm{SU}(\hat\chi)=u_f\cdot{\mathfrak{g}(\theta_\phi)}/{p^n}\]
we see that $u_\mathrm{SU}(\hat\chi)=\pm u_f$, so $u_\mathrm{SU}(\hat\chi)$ 
is a $p$-adic unit, which depends on $\chi$ only up to a sign. Therefore we have: 
\[\pi_f^\mathrm{ac}(\mathcal L^{S}_{\mathbf{f},K })(\chi)=
\frac{u_\mathrm{SU}(\hat\chi)\cdot p^{n(2k-3)}\cdot \Gamma(k-1)^2}{\alpha_p^{2n}}
\cdot \frac{L^S_K(f_0,\chi,k-1)}{\pi^{2k-2}\cdot\Omega_{f_0}^+\cdot\Omega_{f_0}^-} .\]
Recall that, by Proposition \ref{congruence}, 
\[\eta_f=\frac{\langle f_0,f_0\rangle_N}{\Omega_{f_0}^+\cdot\Omega_{f_0}^-}\]
up to $p$-adic units. 
Comparing these formulas, 
we obtain 
\[{\pi_f^\mathrm{ac}(\mathcal L^{S}_{\mathbf{f},K })}(\chi)={{u}(\chi)}\cdot\eta_f\cdot L^\mathrm{HBL}_{S,f,k-1}(\chi)\]
for some unit $u(\chi)$ which depends on $\chi$ only up to sign. We now use the same argument as in the proof of Theorem \ref{SU-CH}. We have infinitely many Hecke characters $\chi$ as above such that 
$u(\chi)=\epsilon$ for at least one choice of $\epsilon\in\{\pm1\}$, thus $u(\chi)$ 
is a unit independent of $\chi$, for $\chi$ in an infinite set of characters.
The values ${\pi_f^\mathrm{ac}(\mathcal L^{S}_{\mathbf{f},K })}(\chi)$ and ${{u}(\chi)}\cdot\eta_f\cdot L^\mathrm{HBL}_{S,f,k-1}(\chi)$
are then equal for infinitely many anticyclotomic Hecke characters $\chi$, and therefore, 
using the Weierstrass preparation theorem, we see that 
${\pi_f^\mathrm{ac}(\mathcal L^{S}_{\mathbf{f},K })}(\chi)$ and ${{u}(\chi)}\cdot\eta_f\cdot L^\mathrm{HBL}_{S,f,k-1}(\chi)$ are equal in $\mathcal{O}_L\pwseries{\Gamma_K^-}$, completing the proof.
\end{proof}

\section{Quaternionic two-variable $p$-adic $L$-functions}\label{quaternionic section}
As in $\S$\ref{subsec:skinner-urban-hida-families}, let $\mathbf{f} \in \mathbb{I} \pwseries{ q }$ be the primitive branch of the Hida family of tame level $N$, fixed as in \S \ref{subsec:main-result}. 
In \cite{LV-MM}, the second-named author and Vigni introduced a two-variable anticyclotomic $p$-adic $L$-function
$L^\mathrm{LV}_\mathbb{I}\in \mathbb{I}\pwseries{\Gamma_K^-}$ by means of big Gross points, an analogue in the definite setting of big Heegner points first introduced by Howard \cite{howard-invmath}. 
The function $L^\mathrm{LV}$ has been further studied in \cite{cas-longo} and \cite{CKL}, and we recall now some of its properties.

Let 
$\phi:\mathbb{I}\rightarrow\mathcal{O}$ be an arithmetic morphism corresponding to the $p$-stabilization $f$ of a $p$-ordinary newform $f_0  \in S_k(\Gamma_0(N))$, and
set \[L^\mathrm{LV}_f=\pi_f(L^\mathrm{LV}_\mathbb{I})\in \mathcal{O}\pwseries{\Gamma_K^-},\] where 
$\pi_f:\mathbb{I}\pwseries{\Gamma_K^-}\rightarrow\mathcal{O}\pwseries{\Gamma_K^-}$ is the identity map on $\Gamma_K^-$ and the map 
$\phi$ on $\mathbb{I}$. Then by \cite[Theorem 3.14]{CKL} we have $(L^\mathrm{LV}_f)=(L^\mathrm{CH}_f)$. 
For each prime $v\mid S$ in $K$,  
define the Euler factor 
\[\mathbb{E}_v(X)=\det\left(\mathrm{Id}-X\frob_v|(\mathbb{T}^\dagger)^{I_v}\right)\]
in $\mathbb{I}[X]$, where $\frob_v$ is a geometric Frobenius at $v$, $I_v$ is the inertia subgroup at $v$, and $\mathbb{T}^\dagger=\mathbb{T}\otimes\Theta^{-1}$ is the central critical twist of Hida's big Galois representation $\mathbb{T}$ attached to $\mathbf{f}$. 
{Then set $\mathbb{E}_v=\mathbb{E}_v(q_v^{-1})$, where $q_v$ is the cardinality of the residue field at $v$,} and define $\mathbb{E}_\ell=\prod_{v\mid \ell}\mathbb{E}_v$ and 
$\mathbb{E}_S=\prod_{\ell\in S}\mathbb{E}_\ell$. Finally define  
\[L^\mathrm{LV}_{S,\mathbb{I}}=\mathbb{E}_S\cdot L^\mathrm{LV}_{\mathbb{I}}.\] 

Theorem \ref{Hida-CH} and Theorem \ref{HBL-SU} show that 
\begin{equation}\label{SU-LV1}
(L^\mathrm{SU}_{S,f})=
((\eta_{f_0}/\eta_{f_0,N^-})\cdot L^\mathrm{LV}_{S,f}).\end{equation}

Let $\eta_\mathbb{I}$ be Hida's congruence ideal defined in \cite{hida-AJM110}, which satisfies the property that \[\phi(\eta_{\mathbb{I}})=\eta_f = \eta_{f_0},\] up to $p$-adic units, for all $\phi:\mathbb{I}\rightarrow\mathcal{O}$ of weight $k$, level $\Gamma_0(Np)$ and trivial character.  
We also denote $\eta_{\mathbb{I},N^-}$ the congruence ideal relative to the $N^-$-new quotient of the Hida-Hecke algebra, defined in 
\cite[Definition 4.12]{Hsieh}, 
which satisfies the property that \[\phi(\eta_{\mathbb{I},N^-})=\eta_{f,N^-} =\eta_{f_0,N^-},\] up to $p$-adic units, for all $\phi:\mathbb{I}\rightarrow\bar\Q_p$ of weight $k$, level $\Gamma_0(Np)$ and trivial character. The results explained in \cite[\S 7]{Hsieh} show that $\eta_\mathbb{I}/\eta_{\mathbb{I},N^-}$ is a $\mathbb{I}$-adic unit under Assumptions \ref{ass}. We thus have in particular that 
\begin{equation}\label{SULV2}
(L^\mathrm{SU}_{S,f})=(L^\mathrm{LV}_{S,f}).\end{equation}

\begin{theorem}\label{main} Under Assumption $\ref{ass}$, 
$(L^\mathrm{SU}_{S,\mathbb{I}})=(L^\mathrm{LV}_{S,\mathbb{I}})$ as ideals in 
$\mathbb{I}\pwseries{\Gamma_K^-}$.  
\end{theorem}

\begin{proof} If one of $L^\mathrm{SU}_{S,\mathbb{I}}$ or $L^\mathrm{LV}_{S,\mathbb{I}}$ is a unit, 
the result follows easily from \eqref{SULV2} combined with \cite[Lemma 3.2]{SU}, so assume in the following that no one of $L^\mathrm{SU}_{S,\mathbb{I}}$ and $L^\mathrm{LV}_{S,\mathbb{I}}$ is a unit.  

Recall the following general form of the Weierstrass preparation theorem (\cite[\S 3, Proposition 6]{Bourbaki}).  
Let $R$ be a complete local ring with maximal ideal $\mathfrak{m}_R$. 
We say that a polynomial $g\in R[X]$ is \emph{distinguished} if $g$ is monic of degree $d$ and all the coefficients of $X^i$ are in $\mathfrak{m}_R$ for all $i<d$. 
Define 
$\Lambda= R[\![X]\!]$, and let 
 $f=\sum_{i\geq 0}a_iX^i\in \Lambda$. 
Suppose that $a_d$ is a unit in $R$ and $a_0,a_1,\dots,a_{d-1}\in \mathfrak{m}_R$  for some integer $d\geq 0$; if this condition is satisfied, we write  $f\not\in\mathfrak{m}_R$ and put $\deg(f)=d$. 
Then there there exist unique elements $u, g\in \Lambda$ such that $u$ is a unit in $\Lambda$ and $g$ is distinguished polynomial of degree $d$ such that $f=ug$. 

We apply the Weierstrass preparation theorem recalled above with $R=\mathbb{I}$ and 
$\Lambda=\mathbb{I}\pwseries{\Gamma_K^-}$. 
First, note that if $L^\mathrm{SU}_{S,\mathbb{I}}$ belong to $\mathfrak{m}_\mathbb{I}$, the maximal ideal of $\mathbb{I}$, then for any arithmetic morphism the $\mu$-invariant of $L^\mathrm{SU}_{S,f}$ would be non-zero, and this is non the case, therefore $L^\mathrm{SU}_{S,\mathbb{I}}\not \in\mathfrak{m}_{\mathbb{I}}$; 
similarly, $L^\mathrm{LV}_{S,\mathbb{I}}\not \in\mathfrak{m}_{\mathbb{I}}$. Since $L^\mathrm{SU}_{S,\mathbb{I}}$ and $L^\mathrm{LV}_{S,\mathbb{I}}$ are not units, then there are distinguished polynomials $G^\mathrm{SU}$ and $G^\mathrm{LV}$ such that $(L^\mathrm{SU}_{S,\mathbb{I}})=(G^\mathrm{SU})$ 
and $(L^\mathrm{LV}_{S,\mathbb{I}})=(G^\mathrm{LV})$.
Pick any arithmetic morphism $\phi:\mathbb{I}\rightarrow\bar\Q_p$; since 
$\phi(G^\mathrm{SU})$ and $\phi(G^\mathrm{LV})$ differ by a $p$-adic unit, and both are monic polynomials, we see that they are equal. We thus see that for each arithmetic morphism $\phi$, we have 
$\phi(G^\mathrm{SU})=\phi(G^\mathrm{LV})$. 
It remains to show that the coefficients of $G^\mathrm{SU}$ and $G^\mathrm{LV}$ are equal, and for this it is enough to show that if
for an element $x\in \mathbb{I}$ we have $\phi(x)=0$ for all arithmetic morphisms $\phi:\mathbb{I}\rightarrow\bar\Q_p$, then $x=0$. For this, let $C=\cap_{\phi}\ker(\phi)$ be the intersection of the kernels $\ker(\phi)$ of $\phi$, where $\phi:\mathbb{I}\rightarrow\bar\Q_p$ varies over all arithmetic morphism. It is clearly enough to show that $C$ is trivial. For each arithmetic morphism $\phi:\mathbb{I}\rightarrow\bar\Q_p$, 
denote $\phi_{|\Lambda_W}:\Lambda_W\rightarrow\bar\Q_p$ its restriction to 
$\Lambda_W$. We first note that the intersection of the kernels
$\ker(\phi_{|\Lambda_W})$ for $\phi: \mathbb{I}\rightarrow\bar\Q_p$ an arithmetic morphism is trivial. To show this, note that 
any morphism $\phi_k:\Lambda_W\rightarrow\bar\Q_p$ which takes $W$ to 
$\zeta_{p^k} - 1$ for $k\equiv 2\pmod{p-1}$, where $\zeta_{p^k}$ is a $p^k$-th primitive root of unity, arises as restriction of an arithmetic morphism $\phi:\mathbb{I}\rightarrow\bar\Q_p$; since one immediately checks that $\cap_{k}\ker(\phi_k)=0$ (where the intersection is over all integers $k\equiv 2\pmod {p-1}$) we also see that $\cap_\phi  \ker(\phi_{|\Lambda_W})=0$. If $C\neq 0$, 
let $\mathfrak{P}\neq 0$ be a height one prime ideal dividing $(C)$. Then by the going down theorem (which holds because $\mathbb{I}$ is integral domain and $\Lambda_W$ is an integrally closed domain) there is an ideal $\mathfrak{p}\neq 0$ such
$\mathfrak{p}=\mathfrak{P}\cap \Lambda_K$. We thus see that 
$\mathfrak{p}\mid \cap_\phi\ker(\phi_{|\Lambda_W})=0$, a contradiction. Hence, $C=0$ and the proof is complete.  
\end{proof}

\bibliographystyle{amsplain}
\bibliography{paper}
\end{document}